\documentclass[10pt]{amsart}

\usepackage{amsfonts,amsmath,amsthm,latexsym,amssymb,marvosym,esint,xfrac,bbm}

\usepackage{graphics, epsfig}
\usepackage{color}
\usepackage{xcolor}
\usepackage{appendix}
\usepackage{verbatim}
\usepackage{ulem}
\usepackage[makeroom]{cancel}
\usepackage[margin=1in]{geometry}
 \usepackage[usenames,dvipsnames]{pstricks}
 \usepackage{pst-grad} 
 \usepackage{pst-plot} 
\allowdisplaybreaks

\usepackage[colorlinks=true,linkcolor=blue,citecolor=red]{hyperref}

\def\R{\mathbb{R}}
\def\N{\mathbb{N}}
\def\dive{\mathrm{div}}
\def\d{\mathrm{d}}
 
\def\dim{\mathrm{dim}}
\def\dist{\mathrm{dist}}
\def\H{\mathcal{H}}
\def\L{\mathcal{L}}
\def\C{\mathcal{C}}
\def\ca{\mathbbmss{1}}

\newtheorem{proposition}{Proposition}[section]
\newtheorem{theorem}{Theorem}[section]
\newtheorem{lemma}{Lemma}[section]
\newtheorem{corollary}{Corollary}[section]
\newtheorem{remark}{Remark}[section]

\numberwithin{equation}{section}
\numberwithin{theorem}{section}
\numberwithin{proposition}{section}
\numberwithin{lemma}{section}
\numberwithin{remark}{section}
\theoremstyle{definition}
\newtheorem{definition}[theorem]{Definition}
\newtheorem{conjecture}[theorem]{Conjecture}

\title{Dimensional lower bounds for contact surfaces of Cheeger sets}

\author{Marco Caroccia }
\address{Dipartimento di matematica, Politecnico di Milano Piazza Leonardo da Vinci 32, 20133 Milano (Mi)}{ }{}
\email{marco.caroccia@polimi.it}

\author{Simone Ciani}
\address{Dipartimento di Matematica e Informatica “U. Dini”, Università degli Studi di Firenze, viale G. Morgagni 67/A, 50134 Firenze, Italy}{}{}
\email{simone.ciani@unifi.it}

\date{\today}

\begin{document}

\begin{abstract}
We carry out an analysis of the size of the contact surface between a Cheeger set $E$ and its ambient space $\Omega \subset \R^d$. By providing bounds on the Hausdorff dimension of the contact surface $\partial E \cap \partial \Omega$, we show a fruitful interplay between this size itself and the regularity of the boundaries. Eventually, we obtain sufficient conditions to infer that the contact surface has positive $(d-1)$ dimensional Hausdorff measure. Finally we prove by explicit examples in two dimensions that such bounds are optimal.
\end{abstract}



\maketitle

\tableofcontents
\addtocontents{toc}{\protect\setcounter{tocdepth}{1}}


\section{Introduction}
The Cheeger problem for a bounded set $\Omega \subset \R^d$ is a typical problem in the Calculus of Variations. It consists in determining the minimum of the functional
\[
\mathcal{F}(E)=  \frac{P(E)}{\L^d(E)}
\] \noindent among all sets $E \subset \Omega$ of finite perimeter. Here $P(E)$ denotes the distributional perimeter of $E$ and $\L^d(E)$ stands for the Lebesgue measure of $E$. This problem has been introduced by Jeff Cheeger in \cite{cheeger1969lower} to bound from below the first nontrivial eigenvalue of the Laplace-Beltrami operator on compact Riemannian manifolds. Lately, this problem received an independent, increasing interest and the associated literature is extremely rich. Some more recent contributions about Cheeger problem on convex sets are \cite{bucur2016faber}, \cite{bucur2018honeycomb},  \cite{caselleschambollenovaga2007}, while non-convex, maximal sets and clusters have been investigated in \cite{buttazzo2007selection}, \cite{caroccia2017cheeger} ,\cite{CarLit17},\cite{caroccia2014note},  \cite{leonardi2017cheeger}, \cite{leonardi2016cheeger},  and finally an approach to optimal regularity can be found in 
\cite{caselleschambollenovaga2010}. This list of contributions is far from being complete and it reflects mostly the interest of the authors toward the specific problem that we are about to describe. Two exhaustive surveys on the Cheeger problem in $\R^d$ can be found in \cite{leonardi2015overview}, \cite{parini2011introduction}, and in references therein. The minimum $h(\Omega)$ of the functional $\mathcal{F}$ (see section \ref{Cheeger-problem}) is called the \textit{Cheeger constant of $\Omega$} and a set $E\subset \Omega$ attaining such a minimum is called a \textit{Cheeger set of $\Omega$}. This isoperimetric constant $h(\Omega)$ can be interpreted also as the \textit{first eigenvalue of the $p$-Laplacian for $p=1$}, and it can be used to give an upper bound to the diameter of Riemannian manifolds with non-negative Ricci curvature (see for instance \cite{Cheeger4diameter}, \cite{kawohl2008p}). Whilst strong regularity properties of the internal boundary $\partial E\cap \Omega$ are known (see section \ref{Cheeger-problem}), the study of contact surfaces of Cheeger sets remained an interesting open problem.
\vskip0.1cm \noindent 
The \textit{contact surface of a Cheeger set} $E \subseteq \Omega$ is the set of points $\partial E \cap \partial \Omega$, where the two boundaries intersect. Given enough regularity on $\partial \Omega$, it is possible to derive a suitable regularity property on $\partial E$ around contact points, but much less is known about how small the set $\partial E\cap \partial \Omega$ can be. The nature of the problem would suggest that every Cheeger set will try to be as big as it can, since the ratio between ${P(rE)}$ and ${\L^d(rE)}$ scales as $1/r$. It is clear consequently that the contact set $\partial E\cap \partial \Omega$ will not be empty. This intuitive argument can be used (see Theorem \ref{thm:reg}) to deduce that $\H^0(\partial E\cap \partial \Omega)$ is greater than two, but it fails to reveal information on higher dimension. Indeed when $\partial \Omega$ has high singularities, such as cusps or angles, the Cheeger set may find convenient to avoid them and prefer a smaller perimeter than a bigger volume. This intriguing behaviour of $E$ complicates the understanding of the size of $\partial E\cap \partial \Omega$ in spite of its simple variational definition. The aim of this work is to fill this lack of knowledge. The answer that we provide is extremely precise and it links the regularity of $\partial \Omega$ to the Hausdorff dimension of the contact surface. We summarize the results of our main Theorems \ref{thm:mainThmCONTACT}, \ref{thm:sharpness} (presented in full generality in Section \ref{Sct:Main}) in the following statement.

\vskip0.2cm

\noindent\textbf{Theorem.} \textit{Let $\Omega \subset \R^d$ be a bounded set and let $E$ be a Cheeger set for $\Omega$. If $\partial \Omega$ has regularity of class $C^{1,\alpha}$ for some $\alpha\in(0,1]$ then   
 \[
 \H^{d-2+\alpha}(\partial E\cap \partial \Omega)>0.
 \] \noindent Moreover, if $\partial \Omega$ has regularity of class $C^{1}$ then  
 \[
 \H^{d-2}(\partial E\cap \partial \Omega)=+\infty.
 \] \noindent Finally, when $d=2$ the previous assertions are sharp in the following sense.
 \begin{itemize}
     \item [a)] There exists an open bounded set $\Omega\subset \R^2$ with Lipschitz boundary and having a Cheeger set $E\subset \Omega$ such that $\H^0(\partial E\cap \partial \Omega)$ is finite.
     \item[b)] For every $\alpha\in (0,1]$ there exists an open bounded set $\Omega\subset \R^2$ with boundary of class $C^{1,\alpha}$ such that $\dim_{\H}(\partial E\cap \partial \Omega)=\alpha$.
 \end{itemize}} \noindent
 It follows as an easy corollary that a boundary regularity of class $C^{1,1}$ ensures the positivity of $\H^{d-1}(\partial E\cap \partial \Omega)$. We show a brief application of our argument to Cheeger sets of convex domains. This application leads us to conclude that if $\Omega$ is convex then the Hausdorff measure of the contact set $\H^{d-1}(\partial E\cap \partial \Omega)$ is positive.\\

 \begin{figure}
     \centering
     \includegraphics[scale=0.8]{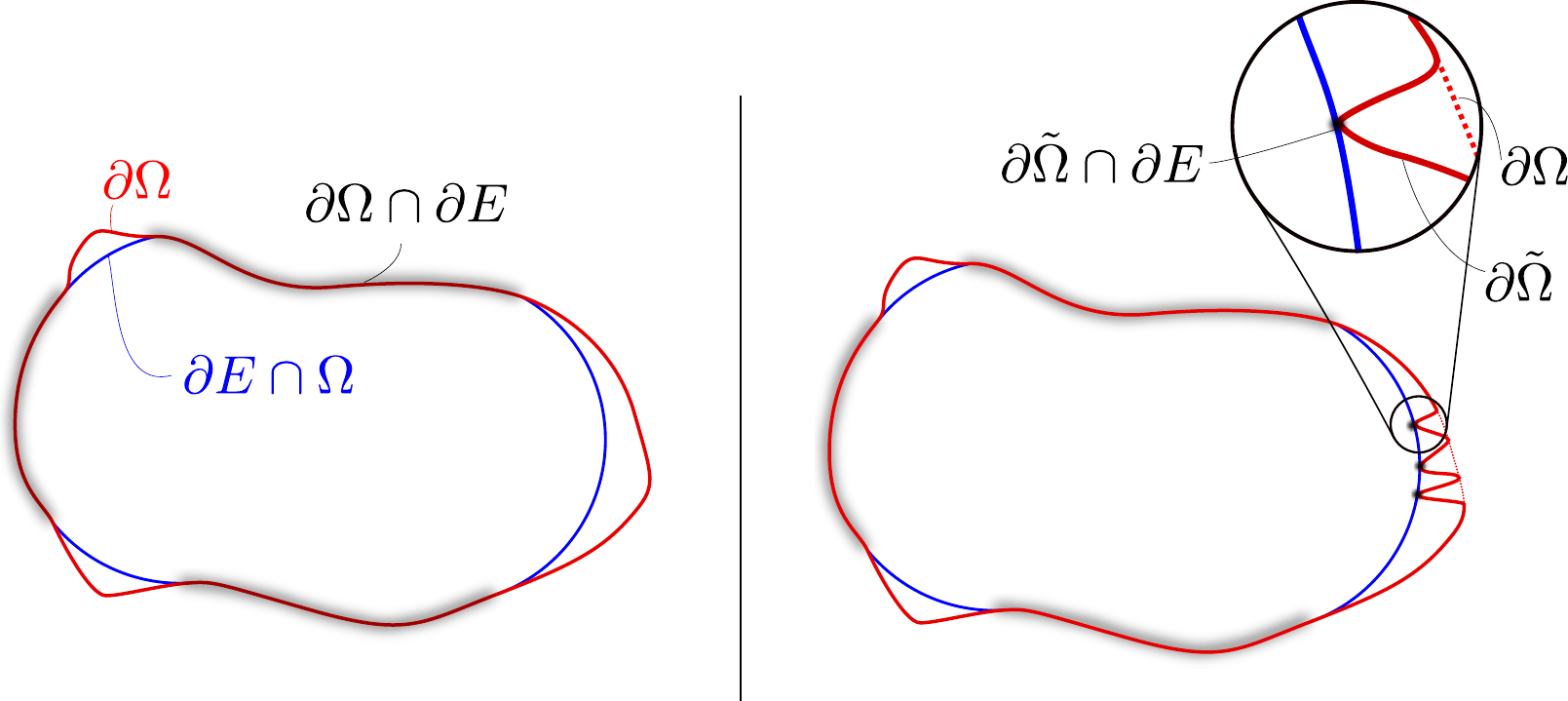}
     \caption{An ambient space $\Omega$ with a representation of one of its Cheeger sets. Provided $\Omega$ has boundary regularity of class $C^{1,1}$ we can ensure the positivity of $\H^{d-1}(\partial E\cap \partial \Omega)$. However there can be part of $\partial \Omega$ where locally the contact set has smaller dimension than $(d-1)$. Consider for example the small inner deformation of $\Omega$ into $\tilde{\Omega}$ producing a locally finite contact set.}
     \label{fig:FINALE}
 \end{figure}
 
 \subsection{Structure of the proof.}
 We start with an open set $\Omega$ with $C^{1,\alpha}$ boundary and we assume by contradiction that $\H^{d-2+\alpha}(\partial E\cap \partial \Omega)$ is zero. We know, from general theory about Cheeger sets, that the internal boundary $\partial E\cap \Omega$ is an analytic hyper-surface with constant mean curvature equal to $h(\Omega)$, apart from a singular set $\Sigma$ whose Hausdorff dimension is at most $d-8$. Now the idea is to extend this property to the whole $\partial E$, so that $E$ is shown to be a set of finite perimeter with constant mean curvature. A refined version of Alexandrov's Theorem (\cite{alexandrov1962characteristic}, \cite{delgadino2019alexandrov}) can be invoked: $E$ must be a finite union of balls. Then, the ambient space $\Omega$ must be a finite union of balls, as well. Now we can show that $\H^{d-2+\alpha}(\partial E\cap \partial \Omega)$ has to be positive, contradicting the hypothesis. The crucial step in the argument consists in the extension of the constant mean curvature property of $\partial E \cap \Omega$ to the whole $\partial E$. Technically, we have a set $E$ with constant distributional mean curvature on $\R^d\setminus (\Gamma\cup \Sigma)$ ($\Gamma=\partial E\cap \partial \Omega$) together with some information on the smallness of $\Gamma$ and $\Sigma$. We would like to combine this information in order to say that $E$ has constant (distributional) mean curvature on the full $\R^d$. The main problem is to remove $\Gamma$.\vskip0.2cm \noindent 
 In order to solve this problem, we make use of the concept of \textit{removable singularities}. We look locally at $\partial E$ as the graph of a function $f_E:A \subset \R^{d'}\rightarrow\R$, $d'=d-1$, which solves the constant mean curvature equation in a set $A \setminus \gamma$. Here $\gamma$ denotes the preimage of $\Gamma$ through the representation (see Section \ref{sct:NOTATION} for precise statements). Roughly speaking, given an open set $A\subset \R^{d'}$ with $d'=d-1$ and a differential operator $L$, a closed set $\gamma\subset A$ is called $L$-removable if, whenever $u$ is a weak solution to the equation
    \[
    Lu=0 \ \ \quad  \text{in $A\setminus \gamma $}
    \]
then $u$ is a weak solution to the equation
    \[
    Lu=0 \ \ \quad \text{in $A$}.
    \]
The literature on this topic is so rich that a complete list would fall outside of the scope of this introduction. For our purposes, we recall the pioneering work developed by Serrin in \cite{serrin1965isolated}, and later in \cite{Serrin2} for more general elliptic equations arising
from bounded conservation laws. On the particular case where $L$ represents the minimal surface equation, we refer to the theorems of De Giorgi-Stampacchia \cite{de1965sulle} and Simon \cite{simon1977theorem}. More recent advances in this topic can be found in \cite{de2004gauss}.\newline 
The typical statement infers that if $\gamma\subset \R^{d'}$ is a closed set with zero $1$-capacity, then $\gamma$ is removable. For instance a condition as $\H^{d'-1}(\gamma)=0$ would be enough to ensure that $\gamma$ is removable for the minimal surface equation. Unfortunately our hypothesis is that $\H^{d-2+\alpha}(f_E(\gamma))=\H^{d'-1+\alpha}(f_E(\gamma))$ is zero, and it does not imply directly that the underlying set $\gamma$ in $\R^{d'}$ has zero $1$-capacity. Therefore we are pushed to study the removability of sets with zero $(d'-1+\alpha)$-dimensional Hausdorff measure. The answer relies on the regularity of the weak solution. Indeed, if $u\in C^{1,\alpha}(A)$ is a weak solution of the Constant Mean Curvature equation in the set $A\setminus \gamma$, then it can be uniquely extended to a solution in the full set $A$, provided $\H^{d'-1+\alpha}(\gamma)=0$ (see Corollary \ref{cor:removability}). This trade-off had been observed by Pokrovskii in \cite{pokrovskii2005removable} for the minimal surfaces equation, in \cite{pokrovskii2005removable2} for the $p-$Laplacian equation and lately in \cite{pokrovskii2005removable3} for uniformly elliptic operators in divergence form. See \cite{pokrovskii2009removableSurvey} for an exhaustive survey on removable sets for elliptic operators in the $C^{1,\alpha}$ class. As we need to remove a singularity of the constant mean curvature equation, a slight adaptation of Pokrovskii's proof in \cite{pokrovskii2005removable} would suffice. 
\vskip0.2cm \noindent Nevertheless, a byproduct of our analysis shows that this trade-off between the size of the singular set and the regularity of the solution is actually a sole property of the divergence operator, that can be performed on a general H\"older continuous vector field. In detail, if we are given a vector field $F\in C^{0,\alpha}(A)$ such that
    \[
    -\dive(F)=g \ \ \ \text{weakly on $A\setminus \gamma$}, \quad \H^{d'-1+\alpha}(\gamma)=0, \quad g\in C^0(\R^{d'}),
    \] then the full equation
 \[
    -\dive(F)=g \ \ \ \text{weakly on $A$}
    \] \noindent is satisfied.
This property of the divergence operator was established by Ponce in \cite{ponce2013singularities}, when $\alpha=0$ and $\H^{d'-1}(\gamma)$ is finite. We state and prove this fact for $\alpha>0$ in Proposition \ref{propo:weakRem} (Section \ref{sct:removable}) and we use this result to deduce an alternative proof of the removability of closed sets in the constant mean curvature equation for $C^{1,\alpha}$ solutions. We think that this approach may have an interest on its own. \newline \noindent The argument is almost concluded: the hypothesis that $\partial E$ is locally the graph of a $C^{1,\alpha}$ function has been achieved in Lemma \ref{lem:impReg}, in the spirit of \cite{giaquinta1981remarks}. The variational inequality lying behind the Cheeger problem (Subsection \ref{sbsbsct:graph}) allows us to use the regularity techniques typical of obstacle problems. 
\vskip0.2cm \noindent Finally we comment the examples built to prove the sharpness of the dimensional bounds. 
The main ingredient required to prove Theorem \ref{thm:sharpness} relies on particular solutions of the constant mean curvature equation in dimension one (see Figure \ref{grafic} and \ref{fig:CANTOR1}).  These solutions solve the ordinary differential equation everywhere but on a closed set $\gamma$, whose prescribed Hausdorff dimension is $\alpha$. Consequently we suitably glue together these solutions in order to obtain a set $E$ which is self-Cheeger and that has constant mean curvature up to a set $\Gamma$ of the chosen Hausdorff dimension. Therefore a suitable family $\{\Omega_{\delta}\}_{\delta>0}$ of ambient spaces can be built around $E$ such that $\partial\Omega_{\delta} \cap \partial E=\Gamma$. The self-Cheeger property and the construction of $E$ ensure now that for some $\delta>0$ there exists a set $\Omega_{\delta}$ that has $E$ as one of its Cheeger sets. Another interesting geometric construction of pathological examples of Cheeger sets in the plane can be found in \cite{leonardi2018two}.

\begin{remark} 
{\rm The argument that we develop is global, and it does not give information on the local behaviour of $\partial E \cap \partial \Omega$. It is possible that locally the contact set has smaller size. Let us consider for instance the situation described in Figure \ref{fig:FINALE}. Here we start from a nice bounded set $\Omega$, where $\H^{1}(\partial \Omega\cap \partial E)$ is positive, and we deform locally the boundary $\partial \Omega \setminus \partial E$ so as to obtain some set $\tilde{\Omega}\subset \Omega $ whose boundary $\partial \tilde{\Omega}$ touches $\partial E$. If we move $\Omega$ inward until it touches $\partial E$, we do not change its Cheeger constant nor its Cheeger set. With this construction we can produce (locally) a set of contact points which behaves as wildly as we want.}
\end{remark}

\begin{remark}
{\rm The structure of the proof shows that, if $E \subseteq \Omega$ is a Cheeger set satisfying \[\H^{d-2+\alpha}(\partial E\cap \partial \Omega \cap A)=0,\] for an open set $A$ and $\partial E  \in C^{1,\alpha}(A)$, then the set $\partial E \cap A$ must be an analytic hyper-surface with constant mean curvature equal to $h(\Omega)$. This consideration leads us to expect that for every open set $A\subset \R^d$ and $\Omega$ convex either the set $\partial E\cap \partial \Omega \cap A$ is empty, or $\H^{d-1}(\partial E\cap \partial \Omega\cap A)$ is positive. In particular we observe that a recent result in \cite{canete} expresses a further step in this direction.}
\end{remark}

\begin{remark}
{ \rm We stress that this argument is more sensitive to the regularity of $\partial E$ than to the regularity of $\partial \Omega$. In particular if in our main Theorem we replace the assumption on the regularity of $\Omega$ with the same regularity on $\partial E$, then the bounds still hold true.}
\end{remark}

\subsection{Open problems.}
 We are able to prove the optimality of the dimensional bounds just in dimension two. The problem with the dimensional jump is the lack of tools concerning Cheeger sets, precisely Theorem \ref{giogian1} and Theorem \ref{giogian2}. We summarise the state of art on this subject in Table \ref{tabella}. 
\begin{table}
{\small
\center
\begin{tabular}{|c|c|c|} 
\hline
\textbf{Regularity properties of $\Omega$}         &   \textbf{Contact surface}  &  \textbf{Sharpness}  \\
\hline
$\partial \Omega$ Lipschitz  &$\H^0(\partial E\cap \partial \Omega)\geq 2$ & \begin{tabular}{c|c}
 
    if $d=2$ &  \begin{tabular}{c}
    \text{}\\
     There exists an open \\
     bounded set $\Omega$\\
     with Lipschitz boundary and \\ 
     $E\subset \Omega$ Cheeger set such that\\
     $\H^0(\partial E\cap \partial \Omega)<+\infty$\\
     \text{}
     \end{tabular}\\ 
 
\hline
\text{}\\
    if $d\geq 3$  &  No example known\\
    \text{}
\end{tabular}\\

\hline
$\partial \Omega \in C^1$  & $\H^{d-2}(\partial E\cap \partial \Omega)=+\infty $ &  \begin{tabular}{c}  \text{}\\
- \\
 \text{}
 \end{tabular}\\
 
\hline
 
\begin{tabular}{c}
    $\partial \Omega\in C^{1,\alpha}$ \\ for $\alpha\in (0,1)$
\end{tabular}& $ \H^{d-2+\alpha}(\partial E\cap \partial \Omega)>0$& 
\begin{tabular}{c|c}
 
    if $d=2$ &  \begin{tabular}{c}
    \text{}\\
     For any $\alpha\in (0,1)$ \\
     there exists an open\\
     bounded set $\Omega$\\
     with boundary of class $C^{1,\alpha}$ \\ and $E\subset \Omega$ Cheeger set\\
     such that 
     $\dim_{\H}(\partial E\cap \partial \Omega)=\alpha$\\
     \text{}
     \end{tabular}\\ 
 
\hline
\text{}\\
    if $d\geq 3$  &  No example known\\
    \text{}
\end{tabular}\\
 
\hline 

$\partial \Omega\in C^{1,1}$ &$\H^{d-1}(\partial E\cap \partial \Omega)> 0$ & \begin{tabular}{c}  \text{}\\
- \\
 \text{}
 \end{tabular}\\

\hline
$\Omega$ convex &$\H^{d-1}(\partial E\cap \partial \Omega)> 0$ &  \begin{tabular}{c}  \text{}\\
- \\
 \text{}
 \end{tabular}\\
 
\hline 
\end{tabular} \\
 \text{}
}
\caption{The behaviour of the contact surface in dependence of the regularity of $\partial \Omega$.}\label{tabella}
\end{table} \noindent 
Nevertheless, by the same approach, it is still possible to construct solutions that cannot be removed on sets whose prescribed Hausdorff dimension is $(d-2+\alpha)$. We believe that similar examples can be given in generic dimension, thus proving the sharpness of Theorem \ref{thm:mainThmCONTACT}. We formally state this assertion in the following conjecture.
 \begin{conjecture}\label{conj}
 For every $\alpha\in (0,1)$ there exists an open bounded set $\Omega\subset \R^d$ with boundary of class $C^{1,\alpha}$ and such that $\dim_{\H}(\partial E\cap \partial \Omega)=d-2+\alpha$.
 \end{conjecture}
 \noindent 
Finally, some links with the prescribed mean curvature problem are worth to be mentioned. Observe that the approach for the Hausdorff bounds produces also some information on the structure of the boundary of $C^{1,\alpha}$ sets with \textit{almost-constant mean curvature} (see Remark \ref{rmk:Self-Cheeger} for details). A similar analysis seems possible for the \text{prescribed mean curvature problem}, of which the Cheeger problem is a particular case. This problem consists in determining the set $E$  attaining the infimum
    \begin{equation}\label{ksk}
    h_{H}(\Omega):=\inf\left\{P(E)+H \L^d(E) \ | \ E\subseteq \Omega \right\}.
    \end{equation}
The internal boundary of the solution of this problem will have a constant mean curvature equal to $H$. In the works \cite{leonardi2018prescribed}, \cite{leonardi2019minimizers}, \cite{leonardi2020rigidity}, recent advances in this problem have been obtained. Notably in the paper \cite{leonardi2019minimizers} similar tools have been developed for the prescribed mean curvature problem, thereby making accessible the geometric construction of sharp examples.

\subsection{Plan of the paper.} In Section \ref{sct:NOTATION} we introduce the required notations and definitions. In Section \ref{Sct:Main} we collect the statements of our main results. This section collects also a brief summary in Table \ref{tabella}. This summary shows the behaviour of the contact surface in dependence of the regularity of $\partial \Omega$. In Section \ref{sct:removable} we develop the analysis of removable singularities for the divergence operator. Lately in Section \ref{ref:pfTHMs} we give the proof of the dimensional bounds. Finally in Section \ref{ref:pfExa} we explicitly construct for each choice of $\alpha\in (0,1)$ some two-dimensional Cheeger sets having finite Hausdorff dimension $\H^{\alpha}(\partial E\cap \partial \Omega)$ .\\
 
\subsubsection*{Acknowledgements}
 The work of MC has been supported by the grant "PRIN 2017 Variational methods for stationary and evolution problems with singularities and interfaces". MC is a member of INdAM-GNAMPA and has been partially supported by the INdAM-GNAMPA Project 2020 "Problemi isoperimetrici con anisotropie" (n.prot. U-UFMBAZ-2020-000798 15-04-2020). The work of SC is partially founded by INdAM (GNAMPA). The authors thank professor Francesco Maggi for the fruitful discussions concerning Proposition \ref{Prop:weakAlFINAL} and its link with Alexandrov's Theorem revisited \ref{thm:AlexRef}. The authors would also like to acknowledge Prof. Giuseppe Buttazzo for the precious discussions concerning the examples in Section \ref{ref:pfExa}. Finally, the authors are grateful to professor Vincenzo Vespri for his remarks and to Dott. Fulvio Gesmundo for the careful reading of the paper. This work has been developed during the COVID-19 lockdown that started in Italy in February 2020.
\section{Notation and preliminaries}\label{sct:NOTATION}

\subsection{General notation}
In the sequel we denote by $\Omega \subset \R^d$ an open bounded set and by $B_r(x)$ the ball of radius $r$ centered at $x$. Similarly $Q_r(x)$ denotes the cube of edge $r$ and centered at $x$. When $r=1,x=0$ we will simply write $B$ or $Q$. Write $\mathbb{S}^{d-1}$ for the $(d-1)$-dimensional sphere in $\R^d$, and $e_1,\dots,e_d$ for the canonical basis of $\R^d$. For $\nu\in \mathbb{S}^{d-1}$ define
    \[
    \nu^{\perp}:=\{z\in \R^d \ | \ z\cdot \nu=0\}.
    \]
  $\L^d$ is the $d$-dimensional Lebesgue measure and we set $\omega_d:=\L^d(B)$. We will usually write $\ca_{A}$ to denote the characteristic function of a set $A$:
    \begin{equation*}
        \ca_{A}(x):=\left\{ \begin{array}{ll}
            1 & \text{if $x\in A$}  \\
            0 & \text{if $x\notin A$}.
        \end{array}\right.
    \end{equation*}

\noindent Let $ Q^{\nu}_r$ be the cube in the space $\nu^{\perp}$ with edge $r$ and   define
    \[
    D^{\nu}_r(R)=\{y+t\nu \ | \ y\in Q^{\nu}_r, \  t\in (-R,R)\}
    \]
the parallelepiped oriented in the direction $\nu$ and having base $Q^{\nu}_r$. When we assume $\nu=e_d$, we omit the index $\nu$ and we write simply $D_r(R), Q_r$.
\begin{definition}[Boundary regularity for sets in $\R^d$]\label{def:BR}
Let $A\subset \R^d$ be a Borel set, $k \in \N, \alpha\in [0,1)$. We say that $\partial A$ has regularity of class $C^{k,\alpha}$ in a neighbourhood of a point $x\in \partial A$ if there exists $\nu\in \mathbb{S}^{d-1}$, $r,R>0$ and a function $f_A:Q^{\nu}_r \rightarrow (-R,R)   $, such that  $f_A\in C^{k,\alpha}(Q^{\nu}_r)$,
        \[
       |f_{A}(y)|<R \ \ \ \text{for all $y\in Q^{\nu}_r$}
        \]
    and  
    \begin{align*}
    (A-x)\cap  D^{\nu}_r(R)&=\{ z\in D^{\nu}_r(R)  \ | \ z\cdot \nu \leq f_A(\mathbf{p}_{\nu}(z)) \},\\
(\partial A-x)\cap D^{\nu}_r(R) &=\{ y+f_A(y)\nu \ | \ y\in Q^{\nu}_{r}  \},
    \end{align*}
being $\mathbf{p}_{\nu}:\R^d \rightarrow \nu^{\perp}$ the projection onto $\nu^{\perp}$. We call $(f_A, D_r^{\nu}(R))$ the \textit{graph representation of $\partial A$}. We do not specify the dependence of $f_A(y)=f_A^x(y)$ on the point $x$ to lighten the notation, the above dependence being clear from the context.\vskip0.1cm 
\noindent Given a Borel set $A\subset \R^d$ and an open set $U\subseteq \R^d$, we say that $A$ has boundary of class $C^{k,\alpha}$ in $U$ and we write $\partial A\in C^{k,\alpha}(U)$, if for every $ x\in \partial A \cap U$ the set $A$ has boundary of class $C^{k,\alpha}$ in a neighbourhood of $x$.\vskip0.1cm
\noindent Finally, we say that a Borel set $A\subset \R^d$  has boundary of class $C^{k,\alpha}$ and we simply write $\partial A\in C^{k,\alpha}$, if for every $x\in \partial A$, there is a neighbourhood $U_{x}$ where the previous definition applies. Similar taxonomy pertains sets with Lipschitz boundary.\vskip0.1cm \noindent 
We say that $M\subset \R^d$ is a $C^{k}$ hyper-surface (respectively analytic) if every point $y\in M$ has a neighborhood expressed by a graph representation $(f_M,D^{\nu}_{r}(R))$ such that $f_M\in C^k(Q_r^{\nu})$ (respectively $f_M$ analytic), through the expression
    \[
    (M-x)\cap D_{r}^{\nu}(R)=\{y+f_M(y)\nu \ | \ y\in Q^{\nu}_r\}.
    \]
\end{definition}
 \subsection{Hausdorff measure and Hausdorff dimension}\label{sbsct:HM}
Given $d\in \N$, $s\in [0,d]$ we define the \textit{Hausdorff pre-measure of step     $\delta$ on $\R^d$ of a set $E\subset \R^d$,} the number
    \begin{equation}\label{eqn:preHM}
    \H^s_{\delta}(E):=\inf\left\{\left. \sum_{j=1}^{+\infty} \omega_s \left(\frac{\mathrm{diam}(F_j)}{2}\right)^s \ \right| \ \begin{array}{c}
        \text{$\{F_j\}_{j\in \N}$ is a countable covering of $E$}\\
        \text{with sets $F_j$ s.t. $\mathrm{diam}(F_j)\leq \delta$ }
    \end{array}\right\},    
    \end{equation}
 where $\omega_s$ is a given constant (tipycally chosen to coincide with the measure of the unit $s$-dimensional ball for $s\in \N$, see \cite[Section 3.3]{maggiBOOK}). It can be shown that $\H^{s}_{\delta}$ is an outer measure on $\R^d$ and it is decreasing in $\delta$ (see \cite{maggiBOOK}). Therefore it is possible to define the \textit{$s$-dimensional Hausdorff outer measure} $\H^s$ with the limit
    \begin{equation}\label{eqn:HM}
    \H^s(E):=\sup_{\delta\geq 0}\{\H^s_{\delta}(E)\}=\lim_{\delta\rightarrow 0} \H^{s}_{\delta}(E).
    \end{equation} \noindent 
    Finally the \textit{Hausdorff dimension} of $E\subset \R^d$ can be defined via the following formula
     \begin{equation}\label{eqn:Hdim}
    \dim_{\H}(E):=\sup\{s\geq 0 \ | \ \H^s(E)=+\infty\} = \inf\{s\geq 0 \ | \ \H^s(E)=0\}.
    \end{equation}
\begin{remark}
{\rm The Hausdorff dimension of any open set $A\subset \R^d$ is $d$. Moreover, if $M$ is a $k$-dimensional manifold in the usual meaning  $\dim_{\H}(M)=k$.}
\end{remark} \noindent 
The following property of $\H^s$-negligible closed sets will be needed for the aim of removability: its proof is postponed to the appendix. Denote by $\mathcal{Q}_j$ a countable family of cubes of edge $2^{-j}$ yielding the dyadic division of $\R^d$ into a grid.
\begin{proposition}\label{propo:HausMesChar}
Let $N\subset \R^d$ be a closed set such that $\H^s(N)=0$. Then for each $\varepsilon >0$ there exists $j\in \mathbb{N}$ and a dyadic decomposition $\mathcal{Q}_j$ of $\R^d$ such that
    \[
    \sum_{\substack{Q\in \mathcal{Q}_{j}: \\ N\cap Q \neq \emptyset} } 2^{-js}\leq 2\varepsilon.
    \]
\end{proposition}

\subsection{Sets of finite perimeter}\label{sbsct:per}
For a Borel set $E\subset \R^d$ and $\Omega\subset \R^d$ an open set we define \textit{the distributional perimeter of $E$ in $\Omega$ as}
    \begin{align*}
        P(E;\Omega)&=\sup\left\{\left.\int_{E}\dive (T) \d x\  \right|\   T\in C^{\infty}(\Omega;\R^d), \ \|T\|_{\infty}\leq 1\right\}.
    \end{align*}
We say that $E$ is a \textit{set of finite perimeter} if
    \[
    P(E)=P(E;\R^d)<+\infty.
    \]
Thanks to the De Giorgi's structure Theorem \cite{DeGiorgiSOFP1}, \cite{DeGiorgiSOFP2}, for every set of finite perimeter $E$, there exists a $(d-1)$-rectifiable set $\partial^*E\subset \partial E$ and a unitary $\H^{d-1}$-measurable vector field $\nu_E:\partial^* E\rightarrow \mathbb{S}^{d-1}$ such that
    \[
    \int_E \dive(T) \d x = \int_{\partial^* E} \nu_E(x)\cdot T(x) \d \H^{d-1}(x) \ \ \ \text{for all $T\in C^{\infty}_c(\R^d)$},
    \]
where $\H^{d-1}$ is the $(d-1)$-dimensional Hausdorff measure (see Subsection \ref{sbsct:HM}). The set $\partial^* E$ is called \textit{reduced boundary of $E$}, and it satisfies \[ P(E;\Omega)=\H^{d-1}(\partial^* E\cap \Omega)=:\H^{d-1}\llcorner_{ \partial^* E}(\Omega).\] 
When the topological boundary is regular enough, the reduced boundary coincides with $\partial E$ itself. A sufficient condition is $\partial E \in C^1$. We say that a set $E$ has distributional mean curvature $H_E$ if there exists a locally summable $\H^{d-1}$- measurable map $H_E:\partial^* E\rightarrow \R$ such that
      \[
    \int_{\partial^* E} \dive_E(T) \d \H^{d-1}(x) = \int_{\partial^* E} (\nu_E(x)\cdot T(x)) H_E(x) \d \H^{d-1}(x) \ \ \ \text{for all $T\in C^{\infty}_c(\R^d)$},
    \]
where
    \[
    \dive_E(T)(x):=\dive(T)(x)- \nu_E(x) \nabla T(x) \nu_E(x),
    \]
defined $\H^{d-1}$-a.e. on $\partial^* E$. We make frequent use also of the \textit{tangential gradient} of a function $u \in C^{\infty}(\Omega)$,
    \[
    \nabla^E u(x):= \nabla u - (\nu_E(x)\cdot \nabla u(x))\nu_E(x) \ \ \ \text{on}\ \  \partial^* E.
    \]
\begin{definition}\label{def:CMC}
Given a set of finite perimeter $E$, we say that $E$ has constant distributional mean curvature $H$ in an open set $U$ and we write
    \[
    H_E=H \ \ \text{distributionally on $U$,}
    \]
if
\[
\int_{\partial^* E } \dive_E(T) \d \H^{d-1}(x) = \int_{\partial^* E}  H (\nu_E(x)\cdot T(x))\d \H^{d-1}(x) \ \ \ \text{for all $T\in C^{\infty}_c(U)$}.
\]
\end{definition} \noindent 
When $\partial E \in C^2$ the distributional mean curvature defined above for $\partial E$ coincides with the usual mean curvature for hyper-surfaces up to the multiplicative constant $\sfrac{1}{(d-1)}$.
 \begin{remark} \label{rmk:MCBalls}
 {\rm For a set of finite perimeter $E$ with $\partial E \in C^2$, the distributional mean curvature  $H_E :\partial E\rightarrow \R$ coincides, up to a multiplicative constant $(d-1)$, with the usual definition of mean curvature for hyper-surfaces. To clarify this point consider $E=B_R$ a ball of radius $R$ and test Definition \ref{def:CMC} with $T(x)=\frac{x}{|x|}=\nu_{B_R}(x)$ to get
    \[ \nabla T= \frac{1}{|x|}\left[\mathrm{Id} - \frac{x}{|x|}\otimes \frac{x}{|x|}\right],
    \quad \quad  \text{and} \quad \quad 
    \dive_{B_R}(T)=\frac{(d-1)}{|x|}. \]
Thus
    \begin{align*}
    (d-1) P(B_R) R^{-1}&=\int_{\partial B_R}\dive_{B_R}(T)\d\mathcal{H}^{d-1}(y)=\int_{\partial B_R} (T\cdot \nu_{B_R}(y))H_{B_R}\d\mathcal{H}^{d-1}(y)\\
    &=P(B_R)H_{B_R},
    \end{align*}
yielding
    \[
    H_{B_R}=\frac{d-1}{R}.
    \] 
    }
    \end{remark}
\begin{definition}[Indecomposability for sets of finite perimeter]\label{def:ind}
A set of finite perimeter $E$ is \textit{decomposable} if there exists two sets $E_i$ of finite perimeter and positive measure, such that $\L^d(E_1\cap E_2)=0$ (namely $E_1,E_2$ are \textit{essentially disjoint}), $E=E_1\cup E_2$ and $P(E)=P(E_1)+P(E_2)$. In this case we say that $E_1, E_2$ decompose $E$.  $E$ is \textit{indecomposable} if it is not decomposable. Call $\{E_i\}_{i=1}^k\subset E$ indecomposable components of $E$ if $E_1,\ldots, E_k$ decompose $E$ and each $E_i$ is indecomposable.
\end{definition} \noindent We refer to \cite{ambrosio2000functions}, \cite{maggiBOOK} for more details on these topics. The following is a refined version of the well-known Alexandrov's Theorem, stated for sets of finite perimeter: it will be crucial for our dimensional analysis. The result has been achieved in \cite[Theorem 1, Corollary 2] {delgadino2019alexandrov}.
\begin{theorem}[Alexandrov's Theorem revisited]\label{thm:AlexRef}
Let $E$ be a set of finite perimeter and finite volume. Suppose that there exists a constant $H>0$ such that
    \[
    \int_{\partial^* E} \dive_E (T)\d\H^{d-1}(x)=\int_{\partial^* E} H(T\cdot \nu_E) \d\H^{d-1}(x) \ \ \ \text{for all $T\in C^{\infty}_c(\R^d;\R^d)$.}
    \]
Then $E$ is the union of a finite number of essentially disjoint balls of radius $\frac{d-1 }{H}$.
\end{theorem}
\begin{remark}\label{rmk:Indonecomp}
{\rm In light of Theorem \ref{thm:AlexRef} and with respect to the Definition of indecomposability \ref{def:ind} it is easy to check that an indecomposable set $E$ with constant distributional mean curvature has to be a single ball.}
\end{remark}
\subsubsection{Constant mean-curvature sets}\label{sbsb:REGCMC}
Let $U$ be an open set and for some $H \in \R$ let $u\in H^1(U)$  be a weak solution to the equation
    \[
    -\dive\left(\frac{\nabla u}{\sqrt{1+|\nabla u|^2}}\right)=H \ \ \ \text{on $U$}.
    \]
Then it is a well-known fact (see for instance \cite{Ennio} or Theorem 3.2 of \cite{Ennio+Guido}) that $u$ is an analytic pointwise solution to the equation
  \[
    -\dive\left(\frac{\nabla u(x)}{\sqrt{1+|\nabla u(x)|^2}}\right)=H \ \ \ \text{on $U$}.
    \]
In particular if a set of finite perimeter $E$ satisfies
    \[
    H_E=H \ \ \text{distributionally on $ U$}
    \]
 we can conclude that $\partial^* E\cap U$ is an analytic hyper-surface with constant mean-curvature.
\subsection{The Cheeger problem} \label{Cheeger-problem}
Given an open bounded set $\Omega\subset \R^d$, the \textit{Cheeger constant of $\Omega$} (see for instance \cite{leonardi2015overview}) is 
    \[
    h(\Omega):=\inf_{E\subset \Omega}\left\{\frac{P(E)}{\L^d(E)}\right\},
    \]
where the infimum is taken among the sets of finite perimeter contained in $\Omega$. We collect in the following Theorem some classically known facts about Cheeger sets. For further literature we refer to \cite{gonzalez1981minimal}.
\begin{theorem}\label{thm:reg} 
Let $\Omega\subset\R^d$ be an open bounded set. Then 
the following statements hold:
\begin{itemize}
    \item[(I)] there exists at least one Cheeger set $E$ of $\Omega$.
    \item[(II)] $\partial^* E \cap \Omega$ is an analytic hyper-surface with constant distributional mean curvature equal to $ h(\Omega)$ and the singular set $\Sigma:=(\partial E\setminus \partial^* E)\cap \Omega$ is closed and has Hausdorff dimension at most $d-8$.
    \item[(III)] If $E$ is a Cheeger set of $\Omega$ and $\Omega$ has finite perimeter then $\partial E\cap \Omega$ can intersect $\partial^*\Omega$ only in a tangential way. This means that $\partial E\cap \partial^*\Omega=\partial^*E\cap \partial^* \Omega$ and for all $x\in \partial^* E\cap\partial^* \Omega$ it holds  $\nu_E(x)=\nu_{\Omega}(x)$.
    \item[(IV)] If $\partial \Omega\in C^1$ then  $\partial E$ has regularity of class $C^{1}$ in a neighbourhood of any $x\in \partial E\cap \partial \Omega$;  
    \item[(V)] If $\partial \Omega\in C^{1,1}$ then  $\partial E$ has regularity of class $C^{1,1}$ in a neighbourhood of any $x\in \partial E\cap \partial \Omega$;
    \item[(VI)] If $\Omega$ is convex then there exists a unique Cheeger set $E$. Moreover $\partial E$ has regularity of class $C^{1,1}$ in a neighbourhood of any $x\in \partial E\cap \partial \Omega$;
    \item[(VII)] For every Cheeger set $E$ of $\Omega$ we have $\H^{0}(\partial E\cap \partial \Omega)\geq 2 $.
\end{itemize}
\end{theorem}
\begin{remark}
\rm{ Notice that the value stated in assertion (II) is consistent with the Definition of distributional mean curvature \ref{def:CMC} (cf. with Remark ]\ref{rmk:MCBalls}). }
\end{remark}
\begin{proof}
Assertions (I)-(III) can be found in \cite[Proposition 3.5]{leonardi2015overview}, \cite{parini2011introduction}. Assertion (IV) comes as a consequence of \cite{miranda1971frontiere} and assertions (V), (VI) are shown in \cite[Proposition 4.3, Proposition 5.2]{parini2011introduction}, \cite[Theorem 2]{caselleschambollenovaga2010}.\vskip0.1cm \noindent 
Therefore we just comment on the last one. Let us suppose by contradiction that $\H^{0}(\partial E\cap \partial \Omega)=0$, then $\dist(\partial E,\partial \Omega)>0$ and we can dilate a bit $E$ into $\lambda E\subset \Omega$, $\lambda>1$. Hence we obtain \[
\frac{P(\lambda E)}{\L^ d(\lambda E)}=\frac{P(E)}{\lambda \L^d(E)}<\frac{P(E)}{\L^d(E)},\] \noindent contradicting the fact that $E$ is a Cheeger set. But also, if $\H^{0}(\partial E \cap \partial \Omega)=1$ we can find a small translation $\tau$ such that $E+\tau \subset \Omega$ and with $\dist(\tau+\partial E,\partial \Omega)>0$. Then we could apply again the dilation argument and denying the minimality property of $E$.
\end{proof} \noindent 
In the proof of Theorem \ref{thm:mainThmCONTACT} we will assume $E$ to be indecomposable. The following Proposition states that for each decomposition of the Cheeger set $E$ into indecomposable components, there exists necessarily one of those which is a Cheeger set too. 
\begin{proposition}
\label{rmk:ind}
Let $E\subset \Omega$ and suppose that $E$ has two indecomposable components $E_1,E_2$. Then it holds
    \[
  h(\Omega)=\min\left\{ \frac{P(E_1)}{\L^d(E_1)}, \frac{P(E_2)}{\L^d(E_2)}\right\}.
    \]\end{proposition} \noindent 
\begin{proof} Assume by contradiction that
    \[
     \frac{P(E_1)}{\L^d(E_1)}\geq \frac{P(E_2)}{\L^d(E_2)},
    \]
then we have
    \[
  P(E_2)\L^d(E_2)+P(E_1)\L^d(E_2)\geq P(E_2)\L^d(E_2)+ P(E_2)\L^d(E_1).
  \]
 In particular
  \[
  h(\Omega)=\frac{P(E_2)+P(E_1)}{\L^d(E_1)+\L^d(E_2)}\geq \frac{P(E_2)}{\L^d(E_2)}\geq h(\Omega),
    \]
and thus $E_2$ is a Cheeger set for $\Omega$.
\end{proof}

\begin{definition}[Self-Cheeger sets]
We say that a set of finite perimeter $E$ is a \textit{self-Cheeger set} if it holds \[\frac{P(E)}{\L^d(E)}=h(E).\]
\end{definition} \noindent 
Finally, we state an important property of Cheeger sets, which will be a crucial ingredient for the application of Proposition \ref{Prop:weakAlFINAL} to the proof of the Theorem \ref{thm:mainThmCONTACT}. For the sake of completeness, we give its proof in the Appendix.
\begin{lemma}\label{lem:tecnico}
Let $E\subset \Omega$ be a Cheeger set of $\Omega$ with $\H^{d-1}(\partial^* E\cap \Omega)>0$. Let $\Sigma\subset \partial E$ be a closed set with $\H^{d-1}( (\partial^* E\cap \Omega) \setminus \Sigma)>0$. Then there exists a constant $C_0=C_0(\Sigma)>0$ and $r_0=r_0(\Sigma)>0$ such that
    \[
    P(E;B_r(x))\leq C_0 r^{d-1} \ \ \ \text{for all $x\in \Sigma$, $r<r_0$}.
    \]
\end{lemma}

    \subsection{Cheeger problem as an obstacle problem} \label{sbsbsct:graph}
    Let $\Omega$ be an open bounded set with $\partial \Omega\in C^{k,\alpha}$ for some $k\geq 1$, $\alpha\in (0,1)$. Let $E\subset \Omega$ be one of its Cheeger sets. By statement (II) of Theorem \ref{thm:reg} for any $x\in \partial^* E\cap \Omega$ there exists a graph representation $\nu\in \mathbb{S}^{d-1}, r,R>0$, $f_E:Q^{\nu}_r\rightarrow (-R,R)$ with $f_E\in C^{\infty}(Q_r^{\nu})$ analytic and solving the constant mean curvature equation
        \[
        -\dive\left(\frac{\nabla f_E(x)}{\sqrt{1+|\nabla f_E(x)|^2}}\right) = h(\Omega) \ \ \text{for all $x\in Q^{\nu}_r$}.
        \]
    If we pick a point $x\in \partial E\cap \partial \Omega$, by assertions (III)-(IV)-(V) of Theorem \ref{thm:reg} we know that $\nu_E(x)=\nu_{\Omega}(x)$, and we can find graph representations $f_E\in C^{1}(Q^{\nu}_r)$, $f_{\Omega}\in C^{k,\alpha}(Q^{\nu}_r)$ such that, for some $R>0$,
    \begin{align*}
    (E-x)\cap  D^{\nu}_r(R)&=\{ z\in D^{\nu}_r(R)  \ | \ z\cdot e_d \leq f_E(\mathbf{p}_{\nu}(z)) \},\\
(\Omega-x) \cap  D^{\nu}_r(R)&=\{ x\in  D^{\nu}_r(R) \ | \ z\cdot e_d \leq f_\Omega(\mathbf{p}_{\nu}(x)) \},\\
(\partial E-x)\cap   D^{\nu}_r(R) &=\{ y+f_E(y)\nu \ | \ y\in Q^{\nu}_{r}  \},\\
(\partial \Omega-x) \cap  D^{\nu}_r(R) &=\{ y+f_\Omega(y)\nu \ | \ y\in Q^{\nu}_{r}  \}.
    \end{align*}
The isoperimetric properties of $E$ imply that $f_E$ is the solution to the following obstacle problem
\begin{equation}\label{eqn:OBS}
\begin{aligned}
    \int_{Q_r}& (\sqrt{1+|\nabla f_E|^2} - h(\Omega)f_E) \d x=\\
    &\inf \left\{  \int_{Q^{\nu}_r} (\sqrt{1+|\nabla w|^2} - h(\Omega)w)\d x\, \, \bigg| \, \,  w\in H^1(Q^{\nu}_r), \, 
    w=f_E \ \text{on $\partial Q^{\nu}_r$},\, 
    w\leq f_{\Omega}  \, \text{on $Q^{\nu}_r$ }     \right\}. \end{aligned} \end{equation} \noindent This implies that, if we define
 \[\Gamma_{r,R}:=(\partial E-x)\cap (\partial \Omega-x) \cap  D^{\nu}_r(R), \quad \quad
    \gamma_{r,R}:=\{y\in Q^{\nu}_r \ | \ y+f_E(y)\nu \in \Gamma_{r,R}\}=\mathbf{p}_{\nu}(\Gamma_{r,R}),
    \] \noindent then $f_E$ satisfies the following properties
    \begin{equation}\label{upper}
    \left\{
    \begin{array}{rr}
    -\dive\left(\frac{\nabla f_E(x)}{\sqrt{1+|\nabla f_E(x)|^2}}\right)= h(\Omega) \ \ &  \ \ \text{for all $x\in Q^{\nu}_r\setminus \gamma_{r,R}$,}\\
      -\dive\left(\frac{\nabla f_E}{\sqrt{1+|\nabla f_E|^2}}\right)\leq  h(\Omega) \ \ &  \ \ \text{weakly on  $Q^{\nu}_r$,}\\
    f_E(x) < f_{\Omega}(x) \ \ &  \ \ \text{for all $x\in Q^{\nu}_r\setminus \gamma_{r,R}$,}\\
    f_E(x)=f_\Omega(x) \ \ &  \ \ \text{for all $x\in \gamma_{r,R}$.}\\
    \end{array}
    \right.
    \end{equation} 

\subsection{Campanato spaces and regularity theory tools}\label{sbs:REG:CAM}
Given an open bounded domain $A \subset \mathbb{R}^{d}$, a function $u \in L^2(A)$ belongs to the \textit{Campanato space} $\mathcal{L}^{2,\lambda}(A)$ if there exist $C,\tilde{\varrho} >0$ such that for each $x \in A$ and $\varrho \leq \tilde{\varrho}$
\begin{equation} \label{campanato-estimate}
\int_{Q_{\varrho(x)} \cap A} | u-u_{Q_{\varrho}(x) \cap A}|^2 \d y \leq C \varrho^{\lambda},
\end{equation} 
where as usual
\begin{equation} \label{average}
u_{D}= \fint_{D} u \, \d x = \frac{1}{\L^d(D)} \int_{D} u \, \d x.
\end{equation} 
The celebrated results of Campanato in \cite{Campa} give the existence of an isomorphism between $\mathcal{L}^{2,\lambda}$ and $ C^{0,\alpha}$ for $\alpha=\frac{\lambda-d}{2}$. Notably this means that  \[
u \in \mathcal{L}^{2,2\alpha+d} \Rightarrow u\in C^{0,\alpha}.\]
We will use the power of this isomorphism together the following important technical Lemma (\cite{AmbrosioCarlottoMassaccesi}), in order to determine an improvement of regularity in Section \ref{sct:removable}. 

\begin{lemma}\label{magiclemma}
Let $\phi$ be a non decreasing, positive function satisfying for $\sigma,A,B,\varrho_* \ge 0$, $b>a>d$,
\begin{align}
\phi(\varrho_1) &\leq A   \bigg( \frac{\varrho_1}{\varrho_2} \bigg)^{b}  \phi(\varrho_2) + B \varrho_2^{a},\quad \quad \ &\text{whenever $0< \varrho_1 \leq \varrho_2 \leq \varrho_*$}, \label{MagicLemma}\\
\phi(\varrho)&\leq \sigma \varrho^{d}, \ \ \ \ &\text{for all $\varrho\leq \varrho_*$}.\label{MagicLemma2}
\end{align} 
Then the following estimate holds 
\[
\phi(\varrho) \leq C \varrho^{a} \ \ \ \text{for all $\varrho \leq \varrho_*$}
\] 
for a constant $C=C(A,B,\tau,a,b,d,\varrho_*,\sigma)$ but independent of $\phi$.
\end{lemma}


\section{Statement of the main results}\label{Sct:Main}
In this section we introduce our main results concerning  the contact surface of Cheeger sets, in terms of the regularity of $\partial \Omega$. 
\begin{theorem}[Contact points of Cheeger sets]\label{thm:mainThmCONTACT}
Let $\alpha\in [0,1]$ and $\Omega\subset \R^d$ be an open bounded set with $\partial \Omega\in C^{1,\alpha}$. Then every Cheeger set $E$ of $\Omega$ has boundary regularity of class $C^{1,\alpha}$ in a neighbourhood of any $x\in \partial E\cap\partial \Omega$ and
   \begin{equation}\label{eq:estMeasure}
        \H^{d-2+\alpha}(\partial E\cap \partial \Omega)>0.
   \end{equation}
As a consequence
   \begin{equation}\label{eq:estDim}
d-2+\alpha \leq \dim_{\H}(\partial E\cap \partial \Omega)\leq d-1.
\end{equation}
Furthermore, if $\alpha=0$ and $\Omega$ is an open bounded set with $\partial \Omega \in C^1$ then for any $E$ Cheeger set of $\Omega$ we have additionally
    \[
    \H^{d-2}(\partial E\cap \partial \Omega)=+\infty.
    \]
\end{theorem} \noindent 
In dimension $d=2$, we are able to provide the following examples saturating the estimates \eqref{eq:estMeasure}, \eqref{eq:estDim}.

\begin{theorem}[Sharpness of the dimensional estimate for $d=2$] \label{thm:sharpness}
There exists an open bounded set $\Omega\subset \R^2$ with Lipschitz boundary having a Cheeger set $E$ such that
    \[
    \H^0(\partial E\cap \partial \Omega)<+\infty.
    \]
Moreover, for any $\alpha\in (0,1)$ there exists an open bounded set $\Omega\subset \R^2$ with boundary regularity of class $C^{1,\alpha}$ having a Cheeger set $E$ such that
    \[
    0<\H^{\alpha}(\partial E\cap \partial \Omega)<+\infty , \ \ \ \ \  \dim_{\H} (\partial E\cap \partial \Omega)=\alpha.
    \]
\end{theorem} \noindent As a consequence of Theorem \ref{thm:mainThmCONTACT} we have that a contact set of full $\H^{d-1}$ measure can be ensured when $\partial \Omega$ has boundary of class $C^{1,1}$ or $\Omega$ is convex.
\begin{corollary}\label{cor:a1}
Let $\Omega\subset \R^d$ be a bounded open set with $\partial \Omega\in C^{1,1}$. Let $E\subset \Omega$ be one of its Cheeger sets. Then
    \[
    \H^{d-1}(\partial E\cap \partial \Omega)>0.
    \] \noindent The same conclusion is valid in case $\Omega$ is convex.
\end{corollary} \noindent

\section{Removable singularities: a result for the divergence operator}\label{sct:removable}

In this Section we study those singularities which are removable for distributional mean curvature operators. In particular we prove the following Theorem, that gives a precise criterion to extend the constant distributional mean curvature in those sets whose Hausdorff measure is small enough.
\begin{theorem}\label{mainthm:local}
Let $E$ be a set of finite perimeter. Let $U$ be an open set and suppose that $\partial E\in C^{1,\alpha}(U)$ for some $\alpha\in [0,1]$. Suppose also that there exists a constant $H\in \R_+$ such that
    \[
    H_E=H \ \ \ \text{distributionally on $ U\setminus \Gamma$}
    \]
according to Definition \ref{def:CMC}, where $\Gamma\subset \R^d$ is a closed set satisfying
\begin{equation}
        \left\{ \begin{array}{ll}
         \displaystyle   \H^{d-2+\alpha}(\Gamma)=0  &  \text{if $\alpha>0$}\\
         \text{}\\
           \displaystyle   \H^{d-2}(\Gamma)<+\infty   & \text{if $\alpha=0$}.
        \end{array}
        \right.
    \end{equation}
 Then  $\partial E\cap U$ is an analytic hyper-surface with constant mean curvature equal to $H$.
\end{theorem} \noindent 
As already explained in the introduction, Theorem \ref{mainthm:local} can be derived by invoking the results in \cite{pokrovskii2009removableSurvey}. We give here an alternative proof, for a more general operator. Indeed, we extend a work of Ponce (\cite{ponce2013singularities}) originally developed for the special case $\alpha=0$. The key idea is that removability on small sets is a property of operators in divergence form. 

\begin{remark}
{\rm The existence result stated in Theorem \ref{thm:sharpness} for $\alpha\in (0,1)$ provides also a partial proof of the sharpness for Theorem \ref{mainthm:local} in $d=2$. Indeed the Cheeger set $E$ of $\Omega$ will have enough regularity and 
    \[
    H_E=h(\Omega) \ \ \ \text{on $\R^d\setminus \Gamma$}
    \]
where $\Gamma=\partial E\cap \partial \Omega$ satisfies $0<\H^{\alpha}(\Gamma)<+\infty$. In particular $\partial E$ is not analytic}.
\end{remark} \noindent 
\begin{proposition}[Theorem 1.1, \cite{ponce2013singularities}]\label{prop:removability}
Let $A\subset \R^d$ be an open bounded set, $\gamma\subset A$ a closed set with $\H^{d-1}(\gamma)<+\infty $ and $F\in C^{0}(A)$ such that
    \[
    -\dive(F)=g \ \ \ \text{weakly on $A\setminus \gamma$},
    \]
with $g\in C(\R^d)$ a continuous function. Then
    \[
    -\dive(F)=g\ \ \ \text{weakly on $A$}.
    \]
\end{proposition} \noindent We extend Proposition \ref{prop:removability} to the case $\alpha>0$ in the next Proposition. 
\begin{proposition}\label{propo:weakRem}
Let $A\subset \R^d$ be an open bounded set, $\alpha\in [0,1]$, $\gamma\subset A$ a closed set with $\H^{d-1+\alpha}(\gamma)=0$ and $F\in C^{0,\alpha}(A)$ such that
    \[
    -\dive(F)=g \ \ \ \text{weakly on $A\setminus \gamma$,}
    \]
with $g\in C(\R^d)$ a continuous function. Then
    \[
     -\dive(F)=g \ \ \ \text{weakly on $A$ }.
    \]
\end{proposition}
\begin{proof}
Let $\varphi\in C^{\infty}_c(A)$ and let us write $K=\mathrm{spt}(\varphi)$, $\gamma'=K\cap \gamma$ which is compact and satisfies $\H^{d-1+\alpha}(\gamma')=0$. We fix $\varepsilon= \varepsilon(\varphi)>0$ and invoke Proposition \ref{propo:HausMesChar} to find $j\in \N$ such that $\mathcal{Q}_j:=\{Q_k^{ j }(x_k^{ j })\}_{k\in \N}$, the countable family of closed cubes centered at $\{x_k^{ j }\}_{k\in \N}$ and of edge $2^{-j}$ tasselling $\R^d$, satisfies
    \[
    \gamma'\subset \bigcup_{\substack{ Q_k^j \in \mathcal{Q}_j :\\  \gamma\cap Q_k^j\neq \emptyset }} Q_k^j, \ \ \   \sum_{\substack{ Q_k^j \in \mathcal{Q}_j :\\  \gamma'\cap Q_k^j\neq \emptyset } } 2^{-j(d-1+\alpha)} \leq  \varepsilon.
    \]
Since $\gamma'$ is compact we can find a finite number of cubes (relabelled) from the grid  $\{Q_1(x_1),\ldots,Q_{k_j}(x_{k_j})\}\in \mathcal{Q}_j $ such that
 \[
    \gamma'\subset \bigcup_{i=1}^{k_j} Q_i(x_i)\subset\subset Q , \ \ \ \   k_j 2^{-j(d-1+\alpha)} \leq  \varepsilon.
    \]
Let us write 
    \[
    \gamma_j:=   \bigcup_{i=1}^{k_j} Q_i(x_i).
    \]
Then we split the equation in two
\begin{equation} \label{split}
    \begin{aligned}
        \int_A F\cdot \nabla \varphi \d x = & \int_{A\setminus \gamma_j} F\cdot \nabla \varphi \d x +\int_{\gamma_j} F\cdot \nabla \varphi \d x \\
        =&\int_{A\setminus \gamma_j} g\varphi \d x+  \int_{\partial \gamma_j}  \varphi (F\cdot \nu) \d\H^{d-1}(x) +\int_{\gamma_j} F\cdot \nabla \varphi \d x
    \end{aligned} \end{equation}\noindent and we estimate the various terms. Last integral in \eqref{split} is controlled from above
    \begin{align*}
        \left|\int_{\gamma_j} (F\cdot \nabla \varphi -g\varphi )\d x\right|\leq \L^d(\gamma_j) (\|F\cdot \nabla \varphi\|_{\infty} +\|g\varphi\|_{\infty})\leq C k_j 2^{-jd} \leq C \varepsilon.
    \end{align*}
 In the following we will denote by $C=C(F,g,\varphi)$ a constant independent of $j,\varepsilon$ and that may vary from line to line. We estimate from above the second integral on the right of \eqref{split} with
 \begin{align*}
        \left|\int_{\partial \gamma_j}  \varphi (F\cdot \nu) \d\H^{d-1}(x)\right|&\leq \sum_{i=1}^{k_j}  \left|\int_{\partial Q_i(x_i) }  \varphi (F\cdot \nu) \d\H^{d-1}(x)\right|\\
         &\leq \sum_{i=1}^{k_j}  2^{-j(d-1)} \left|\int_{\partial Q }  \varphi(x_i+2^{-j}y) (F(x_i+2^{-j}y) \cdot \nu) \d\H^{d-1}(y)\right|.
    \end{align*}
Furthermore we write    
\begin{align*}
    \int_{\partial Q }&  \varphi(x_i+2^{-j}y) (F(x_i+2^{-j}y) \cdot \nu) \d\H^{d-1}(y)\\ 
    =&\sum_{\ell=1}^d\int_{(e_\ell^{\perp} \cap Q)+ \sfrac{e_{\ell}}{2} }  \varphi(x_i+2^{-j}y) (F(x_i+2^{-j}y) \cdot e_{\ell}) \d\H^{d-1}(y)\\
    &-\sum_{\ell=1}^d\int_{(e_\ell^{\perp} \cap Q)-\sfrac{e_{\ell}}{2}}  \varphi(x_i+2^{-j}y) (F(x_i+2^{-j}y) \cdot e_{\ell}) \d\H^{d-1}(y)\\
    =&\sum_{\ell=1}^d\int_{(e_\ell^{\perp} \cap Q)}  \varphi(x_i+2^{-j}y+2^{-j-1}e_{\ell}) (F(x_i+2^{-j}y+2^{-j-1}e_{\ell}) \cdot e_{\ell}) \d\H^{d-1}(y)\\
    &-\sum_{\ell=1}^d\int_{(e_\ell^{\perp} \cap Q)}  \varphi(x_i+2^{-j}y-2^{-j-1}e_{\ell}) (F(x_i+2^{-j}y-2^{-j-1}e_{\ell}) \cdot e_{\ell}) \d\H^{d-1}(y)\\
       =&\sum_{\ell=1}^d\int_{(e_\ell^{\perp} \cap Q)}  G_{j,k}(y;i) \cdot e_{\ell} \d\H^{d-1}(y)
\end{align*}
being  
\begin{align*}
    G_{j,k}(y;i):=&\varphi(x_i+2^{-j}y+2^{-j-1}e_{\ell}) F(x_i+2^{-j}y+2^{-j-1}e_{\ell})\\
    &-  \varphi(x_i+2^{-j}y-2^{-j-1}e_{\ell}) F(x_i+2^{-j}y-2^{-j-1}e_{\ell}) \\
    =&[\varphi(x_i+2^{-j}y+2^{-j-1}e_{\ell}) - \varphi(x_i+2^{-j}y-2^{-j-1}e_{\ell})] F(x_i+2^{-j}y+2^{-j-1}e_{\ell})\\
    &+ \varphi(x_i+2^{-j}y-2^{-j-1}e_{\ell})[F(x_i+2^{-j}y+2^{-j-1}e_{\ell})- F(x_i+2^{-j}y-2^{-j-1}e_{\ell})] \\
 \end{align*} 
 where we do not write explicitly the dependence of $G$ on the index $\ell$. Therefore
    \[
    |G_{j,l}(y;i)|\leq \|F\|_{\infty}\|\nabla \varphi\|_{\infty}2^{-j}+C\|\varphi\|_{\infty}2^{-j\alpha}.
    \]
Notice that the above estimate holds trivially also for $\alpha=0$ with a constant in place of $2^{-j\alpha}$. In particular
\begin{align*}
   \left|\int_{\partial \gamma_j}  \varphi (F\cdot \nu) \d\H^{d-1}(x)\right|&\leq  \sum_{i=1}^{k_j}  2^{-j(d-1)} \sum_{\ell=1}^d \int_{(e_\ell^{\perp} \cap Q)}  |G_{j,k}(y;i)| \d\H^{d-1}(y)\\
           &\leq C k_j (2^{-j d}+2^{-j((d-1)+\alpha)}) \leq C\varepsilon. 
    \end{align*}
Finally we can conclude that 
    \[
   \left|\int_{A} F\cdot \nabla \varphi\d x-\int_A g\varphi\d x\right| \leq C\varepsilon,
    \]
and since $C$ is independent from $\varepsilon,j$ and the above estimate is in force for any fixed $\varepsilon>0$, we obtain for each $\varphi\in C^{\infty}_c (A)$ the validity of the equation
    \[
    \int_{A} F\cdot \nabla \varphi\d x=\int_A g\varphi\d x.
    \]
\end{proof}
\begin{remark}
{\rm The proof shows that a lower regularity can be assumed on the function $g$ in estimating the third integral on the right of \eqref{split}, as for instance $g \in L^1_{loc}(A)$.}
\end{remark} \noindent
A direct application of regularity theory in Subsection \ref{sbsb:REGCMC} and Propositions \ref{prop:removability}, \ref{propo:weakRem} to the vector field $F =\frac{\nabla u}{\sqrt{1+|\nabla u|^2}}$ yields the following Corollary.
   
 \begin{corollary}\label{cor:removability}
    Let $A$ be an open bounded set, $H\in \R$, $\gamma\subset A$ a closed set and $u\in C^{\infty}(A\setminus \gamma)\cap C^{1,\alpha}(A)$ such that
        \[
        -\dive\left(\frac{\nabla u(x)}{\sqrt{1+|\nabla u(x)|^2} }\right)=H \ \ \text{weakly in} \quad A\setminus \gamma.
        \] If
        \begin{equation}
            \begin{cases}
            \H^{d-1+\alpha}(\gamma)=0, \quad \alpha \in (0,1],\\
            \H^{d-1}(\gamma)< +\infty, \quad \alpha=0, 
            \end{cases}
        \end{equation} \noindent then $u\in C^{\infty}(A)$ and
     \[
        -\dive\left(\frac{\nabla u(x)}{\sqrt{1+|\nabla u(x)|^2}}\right)=H \ \ \text{strongly in} \quad A.
        \]
    \end{corollary}
     
\begin{proof}[Proof of Theorem \ref{mainthm:local}]
We prove only the case $\alpha>0$, the case $\alpha=0$ being similar. Firstly, we observe that, since $\partial E \in C^{1,\alpha}(U)$, then $\partial E=\partial^*E$ on $U$. If $\partial E \cap \Gamma= \emptyset$, then $\partial E \cap U$ is an analytic hyper-surface with constant mean curvature equal to $H$, thanks to the know results of Subsection \ref{sbsb:REGCMC}. Otherwise, let $z\in (\partial E\cap \Gamma) \cap U$ and assume without loss of generality $z=0$, $\nu_E(0)=\nu_{\Omega}(0)=e_d$, $\nabla f_E(0)=\nabla f_{\Omega}(0)=0$. For some $r,R>0$  consider $Q_r\subset \R^{d-1}$ and the graph representation $(f_{E},D_r(R))$ of $\partial E$ around $z=0$, with $f_E\in C^{1,\alpha}(Q_r)$ (see Subsection \ref{sbsbsct:graph}, $\nu=e_d$ has been omitted for the sake of shortness). Let us set 
\[ 
 F(y):=y+f_E(y)e_d ,\quad \quad F(y) \in C^{1,\alpha}(Q_r; D_r(R) ),\]
\[\Gamma_{r,R}:= (\partial E\cap \Gamma) \cap D_r(R), \quad \gamma_{r,R}:=\{y\in Q_r \ | \ F(y)\in \Gamma_{r,R}\}.
    \] 
The function $f_E$ is a weak solution to   
 \[
    -\dive\left(\frac{\nabla f_E}{\sqrt{1+|\nabla f_E|^2}}\right)=H \ \ \text{in $Q_r\setminus \gamma_{r,R}$},.
    \]
and consequently by applying results of Subsection \ref{sbsb:REGCMC} we obtain that $f_E\in C^{\infty}(Q_r\setminus \gamma_{r,R})$ is a strong solution to the equation 
 \[
    -\dive\left(\frac{\nabla f_E}{\sqrt{1+|\nabla f_E|^2}}\right)=H \ \ \text{in $Q_r\setminus \gamma_{r,R}$}.
    \]
We observe that $F$ is a Lipschitz function, and the function $F^{-1}:\partial E\cap D_r(R) \rightarrow Q_r$ is Lipschitz too, possibly by decreasing $r,R$ further. In particular (see \cite[Proposition 3.5]{maggiBOOK}) this implies
    \[
    \H^{d-2+\alpha}(\gamma_{r,R})= \H^{d-2+\alpha}(F^{-1}(\Gamma_{r,R}))\leq C \H^{d-2+\alpha}(\Gamma_{r,R})=0.
    \]
We are exactly in the hypothesis of Corollary \ref{cor:removability} and thus we can conclude that $f_E\in C^{\infty}(Q_r)$ is an analytic solution to
\[
    -\dive\left(\frac{\nabla f_E}{\sqrt{1+|\nabla f_E|^2}}\right)=H \ \ \text{in $Q_r$}.
    \]
Consequently, for a small $\varrho>0$, $\partial E\cap B_{\varrho}(z)$  is an analytic hyper-surface with constant mean curvature equal to $H$. By generality of $z \in (\partial E \cap \Gamma) \cap U$, we conclude that $ \partial E   \cap U$ is an analytic hyper-surface with constant mean curvature equal to $H$.
\end{proof}

\section{Proof of Theorem \ref{thm:mainThmCONTACT}: Hausdorff dimension of the contact surface}\label{ref:pfTHMs}
\subsection{Technical tools} In this Section we collect some preliminary tools that are useful for understanding both the local properties of the contact set $\partial E \cap \partial \Omega$ and the interior set $\partial E \cap \Omega$. The former inherits the regularity of $\partial \Omega$ through Lemma \ref{lem:impReg} below, while the latter has a controlled behaviour that we explain with Lemma \ref{lem:tecnico} and which permits the application of an adaptation of Theorem \ref{thm:AlexRef}.

\subsubsection{Regularity of the contact set}\label{sbsct:regimp}
In order to apply Theorem \ref{mainthm:local}, we are concerned with the following improvement of regularity which relies on the interpretation of $\partial \Omega$ as a regular obstacle. 
\begin{lemma}\label{lem:impReg}
Let $\Omega$ be an open bounded set with $\partial\Omega\in C^{1,\alpha}$ for some $\alpha\in [0,1]$ and let $E\subset \Omega$ be one of its Cheeger sets. Then $\partial E \cap \partial \Omega$ has boundary regularity of class $C^{1,\alpha}$.
\end{lemma}

\begin{proof} The cases $\alpha =0,1$ are essentially stated in Assertions (IV) and (V) of Theorem \ref{thm:reg}, so we consider $\alpha \in (0,1)$. We fix $x\in \partial E\cap \partial \Omega$ and we invoke assertion (IV) of Theorem \ref{thm:reg} to deduce that $\partial E$ has $C^1$ regularity in a neighbourhood of $x$ and that $\nu_E(x)=\nu_{\Omega}(x)=\nu$. Assume without loss of generality that $x=0$, $\nu=e_d$, and for $r,R>0$ let $(f_E, D_r(R))$, $(f_{\Omega},D_r(R))$ the graph representations of $\partial E$, $\partial \Omega$ respectively. Then the regularity of $\partial \Omega,\partial E$, at $x=0$ yields $f_{\Omega}\in C^{1,\alpha}(Q_r)$, $f_E\in C^1(Q_r)$. We may assume without loss of generality that
    \[
    f_E(0)=f_{\Omega}(0), \ \ \  \nabla f_E(0)=\nabla f_{\Omega}(0)=0 , \ \ \ f_{\Omega} \ge f_E.
    \]
We will show that for there is a $\varrho_*<\sfrac{r}{2}$ such that for all $x\in Q_{\sfrac{r}{2}}$ and for all $0<\varrho_1\leq \varrho_2 \leq \varrho_*$, $ \beta >\alpha$, and for $A,B>0$ independent of the chosen point $x \in Q_{\sfrac{r}{2}}$,
\begin{equation} \label{Stairwaytoheaven}
\int_{Q_{\varrho_1}(x)} | \nabla f_E- (\nabla f_E)_{Q_{\varrho_1}(x)}|^2 \, \d y \leq A \bigg(\frac{\varrho_1}{\varrho_2}  \bigg)^{2 \beta+d} \int_{Q_{\varrho_2}(x)} | \nabla f_E- (\nabla f_E)_{Q_{\varrho_2}(x)}|^2 \, \d y+ B \varrho_2^{2\alpha+d}.
\end{equation} \noindent 
 Indeed the above, combined with the fact that
    \[
    \int_{Q_{\varrho}(x)} | \nabla f_E- (\nabla f_E)_{Q_{\varrho}(x)}|^2 \d y\leq \sigma \varrho^d, \ \ \sigma:=2\|\nabla f_E\|^2_{L^{\infty}(Q_{r})}
    \]
will allow us, by invoking Lemma \ref{magiclemma} with $b=2\beta+d>2\alpha+d=a>d$, on
 \[
 \phi_x(\varrho):= \int_{Q_{\varrho}(x)} | \nabla f_E- (\nabla f_E)_{Q_{\varrho}(x)}|^2 \d y, \ \ \ \varrho\leq \sfrac{r}{2}
 \] 
 to state that 
    \begin{equation}\label{C}
    \phi_x(\varrho)\leq C \varrho^{2\alpha+d} \ \ \ \text{for all $\varrho\leq \varrho^*$}.
    \end{equation}
The constants $A,B$ in \eqref{Stairwaytoheaven} do not depend on $x\in Q_{\sfrac{r}{2}}$ and consequently the constant $C$ in \eqref{C} above depends on $A,B,\alpha,\beta,d,\varrho_*,\sigma$ but not on $\phi_x$. Therefore the following uniform estimate holds in $Q_{\sfrac{r}{2}}$,
    \[
     \int_{Q_{\varrho}(x)} | \nabla f_E- (\nabla f_E)_{Q_{\varrho}(x)}|^2 \d y\leq C \varrho^{2\alpha+d} \ \ \ \text{for all $\varrho\leq \varrho^*$}.
    \] 
    \noindent 
This implies $\nabla f_E\in \L^{2,2\alpha+d}(Q_{\sfrac{r}{2}})$ and Campanato isomorphism gives the desired regularity  $f_E\in C^{1,\alpha}(Q_{\sfrac{r}{2}})$. Let us divide the proof of estimate \eqref{Stairwaytoheaven} in two main steps.\\
\smallskip

\noindent {\bf Step one:} \textit{Reduction to a variational inequality.} 
We consider the obstacle problem \eqref{eqn:OBS} to which $f_E$ is a solution and we show that this leads to a variational inequality. Let us define the convex set
\[
K= \{w \in H^1(Q_{r})\ | \ w=f_E \, \, \text{on} \, \, \partial Q_{r}, \, \,  w \leq f_{\Omega} \,\,  \text{in} \,\,   Q_{r}  \},
\]
and the energy functional $I:K \rightarrow \R$
\[
I(u)= \int_{Q_{r}} \bigg(\sqrt{1+|\nabla u|^2}-h(\Omega) \, \, u \bigg) \d y
\]
to be minimized in \eqref{eqn:OBS}. Since $E$ is a Cheeger set of $\Omega$, it achieves the infimum of $I$ on $K$, i.e.
\[
I(f_E)= \inf_{ w \in K} 
I(w). \]
For every $u \in K$, $t \in [0,1]$ we have that $f_E-t(f_E-u) ) \in K$ 
and consequently $I(f_E) \leq I(f_E-t(f_E-u))$, yielding the inequality
\[
\frac{\d}{\d t}\Big{|}_{t=0} I(f_E-t(f_E-u)) \geq 0.
\]
Finally, a classical calculation leads us to the variational inequality 
\begin{equation} \label{variational}
\int_{Q_{r}}\bigg{[}\frac{\nabla f_E \cdot \nabla (f_E-u)}{\sqrt{1+|\nabla f_E |^2}} - h(\Omega) (f_E-u)\bigg{]} \, \d y \leq 0, \quad \quad \forall u \in K.
\end{equation}
\smallskip

\noindent \textbf{Step two:} \textit{The comparison technique.} We consider $0<\varrho_1<\varrho_2 \leq \sfrac{r}{2}$, and we split the function $f_E$ into $f_E=U+(f_E-U)$, where $U$ is the solution to the Dirichlet problem
\begin{equation} \label{comparison}
\begin{cases}
\int_{Q_{\varrho_2}(x)} \frac{\nabla U}{\sqrt{1+|\nabla U|^2}} \cdot \nabla \varphi \, \d y=0,&\forall \varphi \in H^1_0(Q_{\varrho_2}(x))\\
(U-f_E) \in H^1_0(Q_{\varrho_2}(x)).
\end{cases}
\end{equation}
By classical results on regularity for solutions to quasilinear equations as \eqref{comparison} we know that $U \in C^{\infty}(Q_{{\varrho_2}}(x))$ and in particular that for all $\varepsilon>0,  \varrho_1\leq \varrho_2$, we have the estimate
\begin{equation} \label{Uregularity}
\int_{Q_{\varrho_1}(x)} | \nabla U- (\nabla U)_{x,\varrho_1}|^2 \, \d y \leq c \bigg(\frac{\varrho_1}{\varrho_2} \bigg)^{d+2-\varepsilon} \int_{Q_{\varrho_2}(x)} | \nabla U- (\nabla U)_{x,{\varrho}_2}|^2 \, \d y,
\end{equation} for a constant $c=c(d)$ that in the sequel may vary from line to line. Write $ f_{x,\varrho}: = f_{Q_{\varrho}(x)}$ to ease notation. We estimate the averaged integral of $\nabla f_E$ with

\begin{align}
\int_{Q_{\varrho_1}(x)} | \nabla f_E- (\nabla f_E)_{x,\varrho_1}|^2 \, \d y \leq & c\left[ \int_{Q_{\varrho_1}(x)}  | \nabla U- (\nabla U)_{x,\varrho_1}|^2 \, \d y \right. \nonumber \\
&\left. +\int_{Q_{\varrho_1}(x)}  | (\nabla U)_{x,\varrho_1}- (\nabla f_E)_{x,\varrho_1}|^2 \, \d y +\int_{Q_{\varrho_1}(x)}  | \nabla (f_E- U)|_{x,\varrho_1}|^2 \, \d y\right] \nonumber\\
  \leq & c \bigg(\frac{\varrho_1}{\varrho_2} \bigg)^{d+2-\varepsilon} \int_{Q_{\varrho_2}(x)} | \nabla U- (\nabla U)_{x,\varrho_2}|^2 \, \d y+ 2c \int_{Q_{\varrho_1}(x)} |\nabla (f_E-U)|^2\, \d y \nonumber\\
 \leq & c \left[ \bigg(\frac{\varrho_1}{{\varrho}_2} \bigg)^{d+2-\varepsilon} \int_{Q_{\varrho_2}(x)} | \nabla f_E- (\nabla f_E)_{x,\varrho_2}|^2 \, \d y +  \int_{Q_{\varrho_2}(x)} |\nabla (f_E-U)|^2\, \d y\right].\label{fEdecomposition}
\end{align}
 \noindent Last inequality is due to the following consideration. Let us show directly that
 \begin{equation} \label{Uestimate}
\int_{Q_{\varrho_2}(x)} | \nabla U- (\nabla U)_{x,\varrho_2}|^2 \, \d y \leq c \left[ \int_{Q_{\varrho_2}(x)} |\nabla f_E- (\nabla f_E)_{x,\varrho_2}|^2 \d y+ \int_{Q_{\varrho_2}(x)} | \nabla f_E - \nabla U|^2\, \d y  \right].
\end{equation}
Firstly we observe that, the $U$ function satisfies to the following equation
\begin{equation} \label{U+c}
\begin{cases}
\int_{Q_{\varrho_2}(x)} \bigg[\frac{\nabla U}{\sqrt{1+|\nabla U|^2}}-\frac{(\nabla U)_{x,\varrho_2}}{\sqrt{1+|(\nabla U)_{x,\varrho_2}|^2}} \bigg] \cdot \nabla \varphi \, \d y=0, & \forall \varphi \in H^1_0(Q_{\varrho_2}(x)),\\
U-f_E \in H^1_0(Q_{\varrho_2}(x)).
\end{cases}
\end{equation} \noindent
Secondly we recall that the map $z\mapsto \frac{z}{\sqrt{1+|z|^2}}$ satisfies the following  fine properties (already observed by \cite{suarez} to prove fine properties of the Fr\'echet derivative of the area functional)
:\\
\begin{enumerate}
    \item[]{\it Monotonicity}: there exists $\nu>0$ such that for all $z_1,z_2 \in \mathbb{R}^n$
\begin{equation} \label{MCO-monotonicity}
\bigg( \frac{z_1}{\sqrt{1+|z_1|^2}}-\frac{z_2}{\sqrt{1+|z_2|^2}} \bigg) \cdot  (z_1-z_2) \ge \nu |z_1-z_2|^2.
\end{equation} \noindent
\item[]{\it Boundedness}: by the Lipschitz character of the map $z \rightarrow z/ \sqrt{1+z^2}$ there exists a $\mu>0$ such that
\begin{equation} \label{MCO-boundedness}
\bigg( \frac{z_1}{\sqrt{1+|z_1|^2}}-\frac{z_2}{\sqrt{1+|z_2|^2}} \bigg) \cdot (a-b) \leq \mu |z_1-z_2| \, |a-b|, \quad \quad z_1,z_2,a,b \in \mathbb{R}^n.
\end{equation}
\end{enumerate}
Now we introduce $\varphi(y) =(U(y)- (\nabla U)_{x,\varrho_2} \cdot y)+((\nabla U)_{x,\varrho_2} \cdot y -f_E(y))$ in equation \eqref{U+c}. By properties \eqref{MCO-monotonicity}, \eqref{MCO-boundedness} and through the use of Young inequality $ab \leq \varepsilon a^p+ C(\varepsilon) b^{p'},\, a,b>0$ when $p,p'$ are conjugate exponents, we obtain \begin{align*}
\int_{Q_{\varrho_2}(x)} | \nabla U- (\nabla U)_{x,\varrho_2}|^2 \, \d y  \leq & c\int_{Q_{\varrho_2}(x)} \bigg[\frac{\nabla U}{\sqrt{1+|\nabla U|^2}}-\frac{(\nabla U)_{x,\varrho_2}}{\sqrt{1+|(\nabla U)_{x,\varrho_2}|^2}} \bigg] \cdot (\nabla U- (\nabla U)_{x,\varrho_2}) \, \d y \\ 
\leq &c \int_{Q_{\varrho_2}(x)} \bigg[\frac{\nabla U}{\sqrt{1+|\nabla U|^2}}-\frac{(\nabla U)_{x,\varrho_2}}{\sqrt{1+|(\nabla U)_{x,\varrho_2}|^2}} \bigg] \cdot (\nabla f_E-(\nabla U)_{x,\varrho_2}) \, \d y \\ 
\leq & c \int_{Q_{\varrho_2}(x)} | \nabla U - (\nabla U)_{x,\varrho_2}|\, |\nabla f_E- (\nabla U)_{x,\varrho_2}| \, \d y \\ 
\leq  &c \bigg{[}\varepsilon \int_{Q_{\varrho_2}(x)} | \nabla U- (\nabla U)_{x,\varrho_2}|^2\, \d y + C(\varepsilon)  \int_{Q_{\varrho_2}(x)} | \nabla f_E- (\nabla U)_{x,\varrho_2}|^2\, \d y   \bigg{]}.
\end{align*} \noindent Hence by choosing $\varepsilon = (2 c)^{-1}$ we can reabsorb the smaller term on the left hand side of the inequality to get
\begin{align*}
\int_{Q_{\varrho_2}(x)} | \nabla U- (\nabla U)_{x,\varrho_2}|^2 \, \d y \leq  & c  \int_{Q_{\varrho_2}(x)} | \nabla f_E- (\nabla f_E)_{x,\varrho_2}+(\nabla f_E)_{x,\varrho_2} - (\nabla U)_{x,\varrho_2}|^2        \, \d y \\
\leq & c \bigg{[} \int_{Q_{\varrho_2}(x)} | \nabla f_E- (\nabla f_E)_{x,\varrho_2}|^2 \, \d y+ \int_{Q_{\varrho_2}(x)} \bigg| \fint_{Q_{\varrho_2}(x)} ( \nabla f_E - \nabla U) \, \d x \bigg|^2 \, \d y  \bigg{]} 
\end{align*} \noindent implying \eqref{Uestimate} by an application of Jensen's inequality to the last term.\\ \\
To accomplish Campanato's inequality \eqref{Stairwaytoheaven}, we estimate the last quantity on the right hand side of the inequality of \eqref{fEdecomposition}. We get back to \eqref{variational} writing $f_E-U+U-v$ instead of $f_E-v$ for $K \ni v = \min \{U, f_{\Omega} \}$, and $U$ solving \eqref{comparison} to have
\[ \int_{Q_{\varrho_2}(x)} \bigg[\frac{\nabla f_E}{\sqrt{1+|\nabla f_E|^2}}- \frac{\nabla U}{\sqrt{1+|\nabla U|^2}}   \bigg] \cdot \nabla(f_E+(-U+U)-v) \d y \leq \int_{Q_{\varrho_2}(x)} h_{\Omega} (f_E-v) \d y \, ,
\] so that monotonicity and boundedness of mean curvature operator together with Young's inequality imply
\begin{equation} \label{hard} \int_{Q_{\varrho_2}(x)} | \nabla (f_E-U)|^2 \d y \leq 
c\int_{Q_{\varrho_2}(x)} |\nabla (U-v)|^2 \d y+ \int_{Q_{\varrho_2}(x)} h_{\Omega} (f_E-v) \d y \,.
\end{equation} \noindent
But $(U-v)$ satisfies the following equation $\forall \varphi \in H^1_0 (Q_{\varrho_2}(x))$
\begin{equation} \label{gdistance}
\begin{split}
\int_{Q_{\varrho_2}(x)} &\bigg( \frac{\nabla U}{\sqrt{1+| \nabla U|^2}}-\frac{\nabla v}{\sqrt{1+| \nabla v|^2}} \bigg) \cdot \nabla \varphi \, \d y \\
&= - \int_{Q_{\varrho_2}(x)} \bigg( \frac{\nabla v}{\sqrt{1+| \nabla v|^2}}-\frac{(\nabla f_{\Omega})_{x,\varrho_2} }{\sqrt{1+| (\nabla f_{\Omega})_{x,\varrho_2}|^2}} \bigg) \cdot \nabla \varphi \, \d y, 
\end{split}
\end{equation}
so that, by inserting in \eqref{gdistance} $\varphi=U-v$ and with the help of Young's inequality again, we obtain
\begin{align*}
\int_{Q_{\varrho_2}(x)} |\nabla(U-v)|^2 \d y  \leq & c \int_{Q_{\varrho_2}(x) \cap \{ f_{\Omega} 
 \leq U \} } \bigg( \frac{\nabla f_{\Omega}}{\sqrt{1+| \nabla f_{\Omega}|^2}}-\frac{(\nabla f_{\Omega})_{{x,\varrho_2}} }{\sqrt{1+| (\nabla f_{\Omega})_{{x,\varrho_2}}|^2}} \bigg) \cdot \nabla (U-v) \d y \\ 
& \leq c\bigg{[}\varepsilon \int_{Q_{{\varrho_2}}(x)} |\nabla(U-v)|^2 \d y+ C(\varepsilon)\int_{Q_{{\varrho_2}}(x)} |\nabla f_{\Omega}- (\nabla f_{\Omega})_{{x,\varrho_2}} |^2 \d y   \bigg{]}.
\end{align*} \noindent
Hence by letting $\varepsilon = (2c)^{-1}$ and exploiting the regularity of the obstacle $f_{\Omega} \in C^{1,\alpha}(Q_r)$ we get finally \begin{equation} \label{U-v}
\int_{Q_{{\varrho_2}}(x)} |\nabla (U-v)|^2 \d y \leq c \int_{Q_{{\varrho_2}}(x)} |\nabla f_{\Omega}-(\nabla f_{\Omega})_{{x,\varrho_2}}|^2 \d y \leq c {\varrho_2}^{d+2\alpha}.
\end{equation}
Last term in \eqref{hard} can be estimated with Poincar\'e and H\" older inequality as
\[
\int_{Q_{\varrho_2}(x)} |f_E-v| \d y \leq c {\varrho_2}^{\frac{d+2}{2}} \bigg{[} \bigg( \int_{Q_{\varrho_2}(x)} | \nabla (f_E-U)|^2 \d y \bigg)^{\frac{1}{2}} +\bigg( \int_{Q_{\varrho_2}(x)} | \nabla (U-v)|^2 \d y \bigg)^{\frac{1}{2}} \bigg{]}, 
\] 
so that for a radius small enough we can reabsorb the first term on the right and obtain the required estimate
\begin{equation} \label{fE-U}
\int_{Q_{\varrho_2}(x)} | \nabla (f_E-U)|^2 \d y \leq c {\varrho_2}^{d+2\alpha}. \end{equation}
Gathering together \eqref{fEdecomposition}, \eqref{fE-U} we obtain \eqref{Stairwaytoheaven}, as desired.
\end{proof}

\subsubsection{Removability of small sets in the interior}\label{sbsct:REM}
In order to deal with the negligible set $\Sigma$ produced by Assertion (II) of Theorem \ref{thm:reg},  we prove the following crucial Proposition. 

\begin{proposition}\label{Prop:weakAlFINAL}
Let $E$ be a set of finite perimeter and $\Sigma \subset \partial E$ be a closed set such that $\H^{d-2}(\Sigma)=0$. Suppose that $E,\Sigma$ have the following properties:
    \begin{itemize}
        \item[a)] There exists $C_0,r_0$ depending on $\Sigma$ such that 
            \[
            P(E;B_r(x))\leq C_0 r^{d-1} \ \ \ \text{for all $x\in \Sigma$, $r<r_0$};
            \]
        \item[b)] $\partial^* E$ has constant distributional mean curvature equal to $H$ on $\R^d\setminus \Sigma$, i.e.
        \[
        \int_{\partial^* E} \dive_E (T)\d \H^{d-1}(x)=\int_{\partial^* E } H (\nu_E(x)\cdot T(x) ) \d\H^{d-1}(x) \ \ \ \text{for all $T\in C^{\infty}_c(\R^d\setminus \Sigma;\R^d)$}.
        \]
    \end{itemize} 
    Then $E$ is a finite union of balls of radius $\frac{d-1}{H}$.
\end{proposition}

\begin{remark}\label{rmk:hardtimes}
{\rm Notice that Proposition \ref{Prop:weakAlFINAL} cannot be deduced by invoking Theorem \ref{mainthm:local} since no information on the regularity of $\partial E$ is given on $\Sigma$ other than property a).}
\end{remark}

\begin{proof}
Let $C_0,r_0$ be the constants given by property a).  By invoking Proposition \ref{propo:HausMesChar}, for any $\varepsilon>0$ we can find a $j\in \N$ and a finite number of cubes $Q_{\varrho_j}(x_1),\ldots,Q_{\varrho_j}(x_{k_j}) \in \mathcal{Q}_j$ of edge length $\varrho_j=2^{-j}$ such that 
    \[
    \Sigma\subset \bigcup_{i=1}^{k_j} Q_{\varrho_j}(x_i) , \ \  k_j \varrho_j^{d-2}\leq \varepsilon.
    \]
Up to further increase $j$ we can also infer that $\varrho_j\leq \frac{r_0}{8d}$. Subordinated to this proof we introduce the short notation
    \[
    U_{s}:=\bigcup_{i=1}^{k_j} Q_{s}(x_i).
    \] 
Let $\zeta_i\in C^{\infty}_c(Q_{2\varrho_j}(x_i))$, $|\zeta_i|\leq 1$ such that
    \begin{equation}\label{control}
  \zeta_i=  \left\{
    \begin{array}{ll}
        1 &  \text{on $Q_{2\varrho_j}(x_i) $ } \\
        0 &  \text{on $Q_{3\varrho_j}(x_i)^c $ }
    \end{array}
    \right.
    \end{equation}
and with $|\nabla \zeta_i|\leq \sfrac{2}{\varrho_j}$. Set now 
 \[
 \zeta(x):=\min_{i\in \N}\{1-\zeta_i(x)\}.
 \]
 Then $\zeta$ is piece-wise smooth and for almost every $x\in \R^d$ satisfies (see for instance \cite{sternberg1999connectivity})
    \begin{align}
    |\nabla \zeta(x)|\leq& \sum_{i=1}^{k_j} |\nabla \zeta_i(x)|\\
    \zeta(x)=&\left\{\begin{array}{ll}
        0 \ \ \ & \text{on $U_{2\varrho_j}$} \\
        1 \ \ \ & \text{on $U_{3\varrho_j}^c$}  .
    \end{array} \right.
    \end{align}
Let $\eta:\R \rightarrow \R $, $\eta \in C^{\infty}_c((0,1))$, $\eta>0$ be a decreasing mollifying kernel such that
    \[
    \int_{\R_+}\eta(t)\d t=1
    \]
and set 
    \[
    \eta_{\delta}(x):=\delta^{-n}\eta\left(\frac{|x|}{\delta}\right), \ \ \  \zeta_{\delta}(x):=(\zeta * \eta_{\delta})(x).
    \]
Let us consider now $\delta<<\varrho_j$ so small that
   \begin{align}
    \zeta_{\delta}(x)=&\left\{\begin{array}{ll}
        0 \ \ \ & \text{on $U_{\varrho_j}$} \\
        1 \ \ \ & \text{on $U_{4\varrho_j}^c$}  .
    \end{array} \right.
    \end{align}
We observe that $\zeta_{\delta}\in C^{\infty}(\R^d)$ and
    \begin{align*}
         |\nabla \zeta_{\delta}(x)|&=0 \ \ \text{on $U_{\varrho_j}\cup U_{4\varrho_j}^c$}\\
         |\nabla \zeta_{\delta}(x)|&\leq \sum_{i=1}^{k_j} |\nabla \zeta_i|* \eta_{\delta}(x) \ \ \ \text{for every $x\in \R^d$}.
    \end{align*}
Moreover, by omitting the center of the cubes,
    \[
    |\nabla \zeta_i| * \eta_{\delta} (x)=0 \ \ \text{on $Q_{\varrho_j}\cup Q_{4\varrho_j}^c$}
    \]
and for $x\in Q_{4 \varrho_j}\setminus Q_{\varrho_j}$
    \begin{align*}
        |\nabla \zeta_i| * \eta_{\delta} (x)&\leq \frac{2}{\varrho_j}.
    \end{align*}
Henceforth, we have 
    \begin{align*}
    \int_{\partial^* E} |\nabla^{E}\zeta_{\delta}(x)|\d\H^{d-1}(x) &\leq C\sum_{i=1}^{k_j}  \int_{\partial^* E } |\nabla \zeta_i | * \eta_{\delta} (x)\d \H^{d-1}(x)\\
    &= C\sum_{i=1}^{k_j}  \int_{\partial^* E\cap (Q_{4\varrho_j}(x_i)\setminus Q_{\varrho_j}(x_i))}  |\nabla \zeta_i | * \eta_{\delta} (x)\d \H^{d-1}(x)\\
       &\leq C\sum_{i=1}^{k_j} \frac{P(E;(Q_{4\varrho_j}(x_i)\setminus Q_{\varrho_j}(x_i) ))}{\varrho_j}.
    \end{align*}
We pick $x_i'\in \Sigma\cap Q_{4\varrho_j}(x_i)$, and being $Q_{4\varrho_j}(x_i)\subset B_{8d\varrho_j}(x_i')$ we use condition $a)$ to get
    \[
     P(E;(Q_{4\varrho_j}(x_i)\setminus Q_{\varrho_j}(x_i) )) \leq P(E; Q_{4\varrho_j}(x_i))\leq P(E;B_{8d\varrho_j}(x_i'))\leq C \varrho_j^{d-1},
    \] for all $i=1,\ldots,k_j$ and a constant $C$ uniform for $x\in \Sigma$. Therefore we get the estimate
    \begin{equation}\label{lalaland}
    \int_{\partial^* E} |\nabla_{\delta}^{E}\zeta(x)|\d\H^{n-1}(x)\leq C k_j\varrho_j^{d-2}\leq C\varepsilon.
    \end{equation}
Analogously we can obtain the inequality
    \begin{align}
     \int_{\partial^* E}|(1-\zeta_{\delta})|\d\H^{d-1}
     &\leq \sum_{i=1}^{k_j} \int_{\partial^* E\cap Q_{4\varrho_j}(x_i)}|(1-\zeta_{\delta})|\d\H^{d-1} \leq \sum_{i=1}^{k_j} P(E;Q_{4\varrho_j}(x_i))\nonumber\\
     &\leq \sum_{i=1}^{k_j} P(E;B_{8d\varrho_j}(x_i'))\leq C k_j \varrho_j^{d-1}\leq C\varepsilon \varrho_j \label{Cuculo}.
    \end{align}
Let $T\in C^{\infty}_c(\R^d;\R^d)$. Then, since $\zeta_{\delta} T\in C^{\infty}_c(\R^d\setminus \Sigma;\R^d)$ we have
    \begin{align*}
        \int_{\partial^* E} \dive_E (\zeta_{\delta} T)\d\H^{d-1}(x)=\int_{\partial^* E} \zeta_{\delta} (T\cdot \nu_E)H\d\H^{d-1}(x).
    \end{align*}
Now the left hand side satisfies
    \begin{align*}
        \int_{\partial^* E} \dive_E (\zeta_{\delta} T)\d\H^{d-1}(x)=&\int_{\partial^* E} (\nabla^E \zeta_{\delta}\cdot T)\d\H^{d-1}(x)+ \int_{\partial^* E} \zeta_{\delta} \dive_E(T)\d \H^{d-1}(x)\\
        =&\int_{\partial^*E} \dive_E(T)\d \H^{d-1}(x)- \int_{\partial^*E} (1-\zeta_{\delta}) \dive_E(T)\d \H^{d-1}(x)\\ 
        &+\int_{\partial^* E} (\nabla^E \zeta_{\delta}\cdot T)\d\H^{d-1}(x)
    \end{align*}
and, due to \eqref{lalaland}, \eqref{Cuculo}
    \begin{align*}
       \left| \int_{\partial^*E} (1-\zeta_{\delta}) \dive_E(T)\d \H^{d-1}(x)-\int_{\partial^* E} (\nabla^E \zeta_{\delta}\cdot T)\d\H^{d-1}(x)\right|&\leq C\varepsilon\left(1+\varrho_j\right)
    \end{align*}
for a constant $C=C(\Sigma,T,d)$. Also, still due to \eqref{Cuculo}
    \[
    \left|\int_{\partial^*E} (1-\zeta_{\delta}) (T\cdot \nu_E)H\d\H^{d-1}(x)\right|\leq C\varepsilon |H| \varrho_j.
    \]
By collecting the above estimates we infer
    \[
    \left| \int_{\partial^* E} \dive_E (T)\d\H^{n-1}(x)-\int_{\partial^*E}  (T\cdot \nu_E)H\d\H^{d-1}(x)\right|\leq C\varepsilon
    \]
for a constant $C=C(\Sigma,H,T,d)$. Being the above valid for all $\varepsilon >0$ we conclude
\[
\int_{\partial^* E} \dive_E (T)\d\H^{d-1}(x)=\int_{\partial^* E}  (T\cdot \nu_E)H\d\H^{d-1}(x)
\]
and the above can be repeated for all $T\in C^{\infty}_c(\R^d;\R^d)$. Finally it is possible to apply Theorem \ref{thm:AlexRef} and conclude that $E$ must be a finite union of balls of radius $\frac{d-1}{H}$.
\end{proof}

\subsection{Proof of Theorem \ref{thm:mainThmCONTACT}.}
We let $\alpha>0$, being the case $\alpha=0$ similar. Taking in consideration the open set $\Omega$ with regularity $\partial \Omega\in C^{1,\alpha}$, we immediately invoke Lemma \ref{lem:impReg} to deduce the regularity $\Gamma= \partial E \cap \partial \Omega \in C^{1,\alpha}$ of the contact set. Moreover, Assertion (II) of Theorem \ref{thm:reg} tells us that $\partial^* E\cap \Omega$ is an analytic hyper-surface of constant mean curvature equal to $h(\Omega)$ and that the singular closed set $\Sigma:=(\partial E\setminus \partial^*E)\cap \Omega$ has Hausdorff dimension at most $d-8$. Suppose by contradiction that 
    \[
    \H^{d-2+\alpha}(\partial E\cap \partial \Omega)=0.
    \]
Since $\Sigma \subset \Omega$ is closed, and due to the regularity of $\partial E$ close to $\partial \Omega$, we can find two open sets $U_{\partial \Omega},U_{\Omega}$ with the following properties
    \begin{itemize}
        \item[(i)] $\partial \Omega \subset U_{\partial \Omega}$, $\Omega\subset U_{\partial \Omega} \cup U_{\Omega}$;
        \item[(ii)] $\Sigma \subset U_{\Omega}$, $U_{\partial \Omega} \cap \Sigma=\emptyset$;
        \item[(iii)] $\partial E\in C^{1,\alpha}(U_{\partial \Omega})$, in particular $\partial E\cap U_{\partial \Omega}=\partial^* E\cap U_{\partial \Omega}$.
    \end{itemize} 
Summarizing, the set $\partial^* E \cap U_{\Omega}$ is an analytic hyper-surface with constant mean curvature equal to $h(\Omega)$ and $E$ has constant distributional mean-curvature on $U_{\partial \Omega}\setminus \Gamma$. By assumption $\H^{d-2+\alpha}(\Gamma)=0$, and since $\partial E\in C^{1,\alpha}(U_{\partial \Omega})$ then Theorem \ref{mainthm:local} applies and $\partial E \cap U_{\partial \Omega}$ is an analytic hyper-surface with constant mean curvature equal to $h(\Omega)$. In particular $E$ has constant mean curvature equal to $h(\Omega)$ on $(U_{\Omega}\cup U_{\partial \Omega})\setminus \Sigma$ which, since $E\subset \Omega\subset \subset U_{\Omega}\cup U_{\partial \Omega}$, is equivalent to say that $E$ has constant mean curvature equal to $h(\Omega)$ on $\R^d\setminus \Sigma$. Moreover the singular set $\Sigma\subset \partial E$ is small enough, so that Property b) of Proposition \ref{Prop:weakAlFINAL} holds. But $E$ is a Cheeger set of $\Omega$ and 
\begin{equation}\label{daichepalle}
        \H^{d-1}(\partial^* E\cap \Omega)>0.
    \end{equation} \noindent
Indeed, if otherwise $\H^{d-1}(\partial^* E\cap \Omega)=0$ then 
    \[
    \H^{d-1}(\partial^* E\cap \partial \Omega)>0 
    \]
and thus trivially
    \[
    \H^{d-2+\alpha}(\partial^* E\cap \partial \Omega)>0.
    \] \noindent 
    Hence it is possible to invoke Lemma \ref{lem:tecnico} to conclude that property $a)$ of Proposition \ref{Prop:weakAlFINAL} holds as well on $E,\Sigma$. This is enough to apply Proposition \ref{Prop:weakAlFINAL} and to conclude that $E$ must be a finite union of balls of radius $\frac{d-1}{h(\Omega)}$. Since we can move from $E$ to one of its indecomposable components (see Remark \ref{rmk:ind}), we can suppose that $E$ is a single ball of radius $ \frac{d-1}{h(\Omega)}$. Accordingly
\[
  h(\Omega)=\frac{P
  (E)}{\L^d(E)}=\frac{d\omega_d \left(\frac{d-1}{h(\Omega)}\right)^{d-1}}{\omega_d\left(\frac{d-1}{h(\Omega)}\right)^{d}}=h(\Omega)\frac{d}{d-1},
    \]
and we bump into a contradiction. This contradiction is a consequence of the fact that we exploited that $\H^{d-2+\alpha}(\partial E\cap \partial \Omega)=0$ to extend the validity of the constant mean curvature equation also on $\partial E\cap \partial \Omega$ by means of Theorem \ref{mainthm:local}.


\begin{remark}\label{rmk:Self-Cheeger}
{ \rm 
The same proof produces information about the size of the region where $E$ does not have constant mean curvature, relatively to its boundary regularity. More in detail, let $E$ be an indecomposable set of finite perimeter with $\partial E\in C^{1,\alpha}$ and define 
\begin{equation}
\mathrm{Cmc}(\partial E;H):=\left\{x\in \partial E \ \left| \ \begin{array}{c}
   \text{there exists $r=r_x>0$ such that}\\
   \text{$B_r(x)\cap \partial E$ is an analytic hyper-surface}\\ 
   \text{with constant mean curvature equal to $H$}
\end{array} \right. \right\}.
\end{equation}
Then, either $E$ is a ball or
    \begin{equation}\label{geppetto}
    \H^{d-2+\alpha}(\partial E\setminus \mathrm{Cmc}(\partial E;H))>0.
    \end{equation}
Indeed if $E$ is not a ball and we violate \eqref{geppetto}, a contradiction follows in a similar fashion to the proof of Theorem \ref{thm:mainThmCONTACT}. }
\end{remark}


\section{Proof of Theorem \ref{thm:sharpness}: building sharp examples in $2$-d}\label{ref:pfExa}
In this section we show a geometric construction of a set $\Omega\subset \R^2$ whose Cheeger set $E$ has $C^{1,\alpha}$ boundary regularity and it is such that 
\[
\dim_{\H}(\partial E\cap \partial \Omega)=\alpha, \quad \quad \H^{\alpha}(\partial E\cap \partial \Omega)<+\infty.
\]
This construction will be done in several steps by using the properties of Cantor staircase-type functions. 
\subsection{Technical tools}
We invoke the following criterion from \cite[Theorem 1.1]{leonardi2017cheeger}, suitably adapted for our purposes.
\begin{theorem}[Self-Cheeger criterion]\label{giogian1}
Let $E\subset \R^2$ be a simply connected open bounded set with Lipschitz boundary such that at any point $x\in \partial E$ there exists a ball $B_r\subset E$ of radius $ r=\frac{\L^2(E)}{P(E)}$ tangent to $x$.\newline 
Then $E$ is self-Cheeger, i.e.
    \begin{equation}
        \frac{P(E)}{\L^2(E)}=\min\left\{\left.\frac{P(F)}{\L^2(F)} \right| \ F\subseteq E \right\}.
    \end{equation}
\end{theorem}
\begin{definition}\label{def:NoNeck}
Let $\Omega\subset \R^2$ be a simply connected open set with $C^1$ boundary. We say that $\Omega$ has no necks of radius $r$ if for any $x,y\in  \Omega $ such that $\dist(x,\partial \Omega),\dist(y,\partial \Omega)>r$ there exists a continuous $C^1$ curve $\gamma:[0,1] \rightarrow \Omega$ joining $x$ and $y$ such that 
    \[
    B_r(\gamma(t))\subseteq \Omega \ \ \text{for all $t\in [0,1]$}.
    \]
\end{definition}

\begin{remark}\label{senzacollo}
{ \rm The property of having no necks of radius $r>0$ is equivalent to the path-connectedness of the {\it inner parallel set}
\[
\Omega^{r}:=\{x\in \Omega \ | \ \mathrm{dist}(x,\partial \Omega)> r\}.
\]}
\end{remark}

\vskip0.2cm \noindent 
Let us denote the Minkowski sum of two sets $A,B\subset \R^N$ with
    \[
    A\oplus B:=\{x+y\ | \ x\in A, \ y\in B\}=\bigcup_{x\in A} (B+x).
    \] \noindent In $\R^2$ there exists a particular characterisation of the Cheeger maximal set of those domains with no neck of radius $r>0$ (see \cite[Theorem 1.4]{leonardi2017cheeger}). 
 \begin{theorem}[Cheeger constant of a domain with no necks]\label{giogian2}
Let $\Omega$ be a simply connected open bounded set with Lipschitz boundary and having no necks of radius $r:=\frac{1}{h(\Omega)}$. Then the maximal Cheeger set is given by
    \[
    E=\Omega^r\oplus B_r=\bigcup_{x\in \Omega^{r}} B_r(x).
    \]
Moreover $r=h(\Omega)^{-1}$ is the unique positive solution to the equation
    \[
    \pi r^2=\L^2(\Omega^r).
    \]
\end{theorem} \noindent 
Let $\Omega \subset \R^2$ be a simply connected domain with Lipschitz boundary and no necks of radius $r$, and let us denote with $\mathcal{M}_0(\Omega^r)$ the Minkowski content of $\partial \Omega^r$, that in this case is finite. The following extension of Steiner formulas used in \cite{leonardi2017cheeger} will be an important tool for our construction:
\begin{align}
    \L^2(\Omega^r\oplus B_r)&=\L^2(\Omega^r)+\mathcal{M}_0(\Omega^r)r +\pi r^2\label{steiner1}\\
    P(\Omega^r\oplus B_r)&=\mathcal{M}_0(\Omega^r) +2\pi r\label{steiner2}.
    \end{align}
\subsubsection{Cantor sets and Cantor staircase properties} \label{cantor}
Let $\tau\in (0,1)$, and let us set 
\[\mathcal{C}_1(\tau) = [0,1]\setminus \left(\frac{1-\tau}{2},\frac{1+\tau}{2}\right).\]\noindent  Then we define $\mathcal{C}_n(\tau)$ as the set obtained from $\mathcal{C}_{n-1}(\tau)$ by removing, on  each of its connected components, the central interval of length $\frac{\tau (1-\tau)^{n-1}}{2^{n-1} }$. We observe that each set $ \mathcal{C}_{n-1}(\tau)$ is made by $2^{n-1}$ disjoint intervals and therefore its length is
    \begin{align*}
    \L^1(\mathcal{C}_n(\tau))&=\L^1(\mathcal{C}_{n-1}(\tau))- \tau (1-\tau)^{n-1}=\L^1(\mathcal{C}_{1}(\tau))-\tau\sum_{i=1}^{n-1}(1-\tau)^{i}
    =(1-\tau)^{n}.
      \end{align*}
The Hausdorff dimension of the limiting Cantor-type set $\C(\tau):=\bigcap_{n=1}^{+\infty} \C_n(\tau)$ is precisely 
    \begin{equation}\label{alf}
    \dim_{\H} (\mathcal{C}(\tau))=\alpha(\tau), \quad \quad \quad \alpha(\tau)=\frac{\log(2)}{\log\left(\frac{2}{1-\tau}\right)}.
   \end{equation}
We refer to \cite{jones2001lebesgue} for details on generalized Cantor functions and to \cite{falconer2004fractal}, \cite{Falconer3} for general theory of fractals. In particular, by varying $\tau\in (0,1)$ we can reach all $\alpha\in (0,1)$.  
We consider the function
    \[
    s_{n,\tau}(t):= \frac{\mathcal{L}^1\left(\mathcal{C}_n(\tau) \cap [0,t] \right)}{(1-\tau)^n}   = \frac{1}{(1-\tau)^n}\int_0^t \ca_{\mathcal{C}_n(\tau)}(r)\d r.
    \]
It is a well-known fact that $s_{n,\tau}$ uniformly converges on $[0,1]$ to a Cantor-type staircase function $s_{\tau}$ which is $C^{0,\alpha}([0,1])\cap C^{\infty}((0,1)\setminus \C(\tau))$ for $\alpha$ identified by \eqref{alf}. We collect here some elementary properties of $s_{\tau}$. More precisely, for fixed $H,\ell \in \R_+$ we consider
    \begin{align*}
        s_{\tau}(t;H,\ell):= H\ell s_{\tau}(\sfrac{t}{\ell}).
    \end{align*}
    \begin{lemma}\label{lem:propSn}
    For any $H,\ell\in \R_+$, $\tau\in (0,1)$ it holds 
    \begin{itemize}
        \item[a)] $s_{\tau}(0;H,\ell)=0$,   $s_{\tau}(\sfrac{\ell}{2};H,\ell)=\frac{H\ell}{2}$, $s_{\tau}(\ell)=H\ell$;
        \item[b)] $s_{\tau}(t;H,\ell)> H t$ for all $t\in (0,\sfrac{\ell}{2})$, $s_{\tau}(t;H,\ell)< H t$ for all $t\in (\sfrac{\ell}{2}, \ell)$ ;
        \item[c)] $s_{\tau}(t;H,\ell)= H\ell - s_{\tau}(\ell-t;H,\ell)$ for all $t\in (0,\sfrac{\ell}{2})$;
        \item[d)] $s_{\tau}(\cdot;H,\ell)\in C^{0,\alpha}((0,\ell)\cap C^{\infty}((0,\ell)\setminus \ell \C(\tau))$ being $\alpha=\alpha(\tau)$ defined as in \eqref{alf};
        \item[e)] $s_{\tau}'(t;H,\ell)=0$ on $C^{\infty}((0,\ell)\setminus \ell \C(\tau))$;
        \item[f)] $|s_{\tau}(t;H,\ell)-Ht|<\frac{H\ell}{2} $.
    \end{itemize}
\end{lemma}
\noindent 
From now on we will consider only those values $H\ell<2$. Let $\tau\in (0,1)$, define the function
    \begin{equation}\label{e:functionu}
   u_{\tau}(t;H,\ell):=\int_0^t \frac{( s_{\tau}(r;H,\ell)- H r)}{\sqrt{1-(s_{\tau}(r;H,\ell)-H r)^2}} \d r.
    \end{equation}

\begin{remark}\label{rmk:ValuesOfUn}
{\rm  By easy manipulations we observe that
    \[
    \frac{u_{\tau}'(t;H,\ell)}{\sqrt{1+(u_{\tau}'(t;H,\ell))^2}} = s_{\tau}(t;H,\ell) - H t 
    \]
and 
    \[
   \sqrt{1+ (u_{\tau}'(t;H,\ell))^2} = \frac{1}{\sqrt{1 - ( s_{\tau}(t;H,\ell) - H t)^2} }.
    \]
    }
\end{remark}
   
\begin{figure}
    \centering
    \includegraphics[scale=0.7]{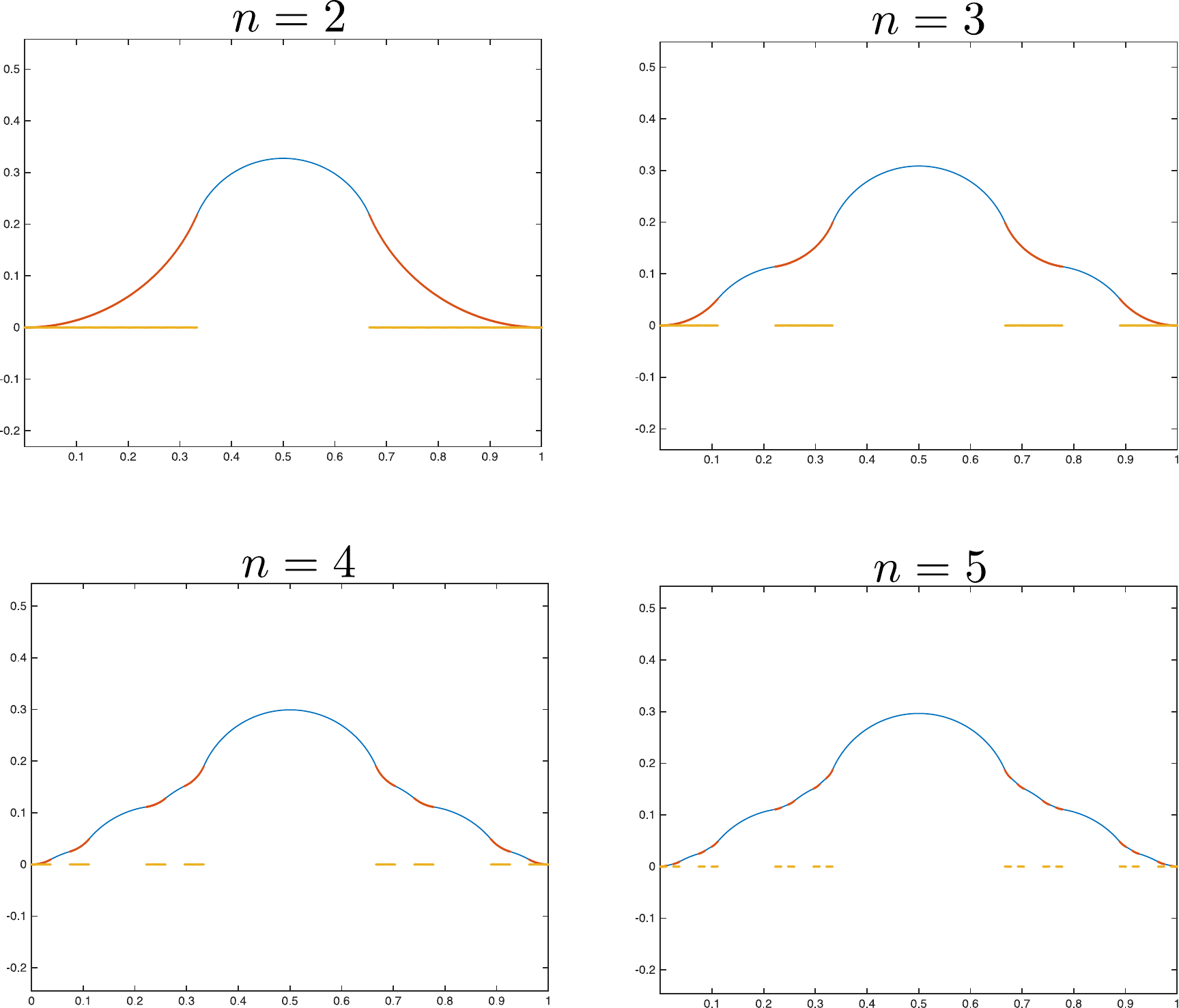}
    \caption{Here are represented the first iterations $u_{n,\tau}$ of the function $u_{\tau}$ for the case $\tau=1/3$. They are defined as the $u_{\tau}$ but with $s_{n,\tau}$ in place of $s_{\tau}$. The Cantor-type set is depicted in yellow while the blue part is the region of constant mean curvature. The orange part lying above the Cantor-type set is the region where the function fails to solve the constant mean curvature ODE. For the case $\tau=1/3$ the function $u_{\tau}$ is well defined up to $H\ell<6$ and, since for small value of $H\ell$ the oscillatory effect is not quite visible, in order to magnify the behaviour of $u_{\tau}$ the parameter has been set to be $H=5.5$, $\ell=1$.}
    \label{grafic}
\end{figure} \noindent 
In the following Lemma we state some properties of $u_{\tau}$ that can be derived from the properties of $s_{\tau}$, and we refer to Figure \ref{grafic} where few iterations are depicted.
    \begin{lemma}\label{lem:propofUn}
    For any $H\ell<2$, $\tau\in (0,1)$ the function $u_{\tau}(\cdot;H,\ell)$ satisfies the following properties:
        \begin{itemize}
            \item[a)] $u_{\tau}(0;H,\ell)=u_{\tau}(\ell;H,\ell)=u_{\tau}'(0;H,\ell)=u_{\tau}'(\ell;H,\ell)=0$;
            \item[b)] $u_{\tau}(t;H,\ell)>0$ for all $t\in (0,\ell)$;
            \item[c)] $u_{\tau}\left(\frac{\ell}{2}-t;H,\ell\right)=u_{\tau}\left(\frac{\ell}{2}+t;H,\ell\right)$ for all $t\in (0,\sfrac{\ell}{2})$;
            \item[d)] $u_{\tau}(\cdot;H,\ell)\in C^{1,\alpha}((0,\ell))\cap C^{\infty}((0,\ell)\setminus \ell \C(\tau))$ where $\alpha=\alpha(\tau)$ is defined as in \eqref{alf};
            \item[e)] $\displaystyle -\left(\frac{u_{\tau}'(t;H,\ell)}{\sqrt{1+(u_{\tau}'(t;H,\ell))^2}}\right)'=H$ for all  $t\in (0,\ell)\setminus \ell \C(\tau)$.
        \end{itemize}
    \end{lemma}
\begin{proof}
Clearly properties d), e) come immediately from Remark \ref{rmk:ValuesOfUn} and from the properties of $s_{\tau}$. To complete the proof we see that it is enough to prove the simmetry relation c), from which a), b) (combined with properties a), b) of Lemma \ref{lem:propSn}) will follow.\newline 
Indeed, by property c) of Lemma \ref{lem:propSn} we have 
\begin{equation*} 
    s_{\tau}(t;H,\ell)-H t =H(\ell-t)-s_{\tau}(\ell-t;H,\ell).
\end{equation*}
which implies
    \begin{equation}\label{eqn:symmetry}
    s_{\tau}\left(\frac{\ell}{2}-t;H,\ell\right)-H\left(\frac{\ell}{2}-t\right)=-\left[s_{\tau}\left(\frac{\ell}{2}+t;H,\ell\right)-H\left(\frac{\ell}{2}+t\right)\right].
    \end{equation}
Hence, setting
    \[
    P(t)=u_{\tau}\left(\frac{\ell}{2}+t;H,\ell\right)-u_{\tau}\left(\frac{\ell}{2}-t;H,\ell\right)
    \]
we have $P'(t)=0$ which gives the desired symmetry on $u_{\tau}$.
\end{proof} \noindent 
Some additional properties of $u_{\tau}$ required to run the construction are contained in the following Lemma.
 \begin{lemma}\label{lem:inequality}
  For any $H\ell<2$, $\tau\in (0,1)$ it holds that         \[
          u_{\tau}(t;H,\ell)-\frac{1}{H} \sqrt{1-(s_{\tau}(t;H,\ell)-Ht)^2}\leq u_{\tau}(r;H,\ell) -\frac{1}{H}\sqrt{1 - \left(s_{\tau}(t;H,\ell)-H r \right)^2} 
        \]
   for all $t\in [0,\ell]$ and for any $r\in (0,\ell)\cap  H^{-1}(s_{\tau}(t;H,\ell) -1,s_{\tau}(t;H,\ell) +1)$. 
    \end{lemma}
\begin{proof}
We omit to specify $H,\ell$ in the argument of $u_{\tau},s_{\tau}$ for the sake of shortness. Notice that the function
\[
P_t(r):=u_{\tau}(r) -\frac{1}{H}\sqrt{1 - \left(s_{\tau}(t)-H r \right)^2} 
\]
is $C^1((0,\ell))$ and
    \[
    P_t'(r)=\frac{\left(s_{\tau}(r)-H r \right)}{\sqrt{1 - \left(s_{\tau}(r)-Hr \right)^2}}-\frac{(s_{\tau}(t)-H r)}{\sqrt{1 - \left(s_{\tau}(t)-Hr \right)^2} }.
    \]
So $P_t'(t)=0$. Moreover, for $r\in ( H^{-1}s_{\tau}(t)-H^{-1},t)$  we have
    \[
    1> s_{\tau}(t)-Hr >s_{\tau}(r)-Hr
    \]
since $s_{\tau}$ is a non decreasing function and $r<t$. Notice that the function
    \[
    z\mapsto \frac{z}{\sqrt{1-z^2}}
    \]
is also non decreasing and thus
    \begin{align*}
    P_t'(r)&\leq  0 \ \ \ \text{on $r\leq t$}.
    \end{align*}
Analogous computation yields also that $P_t'(r)\geq 0$ on $r\geq t$, yielding that $r=t$ is a point of global minimum for $P$ and finishing the proof.
\end{proof}
Next two Propositions are crucial for the construction.
\begin{figure}
    \centering
    \includegraphics[scale=1.5]{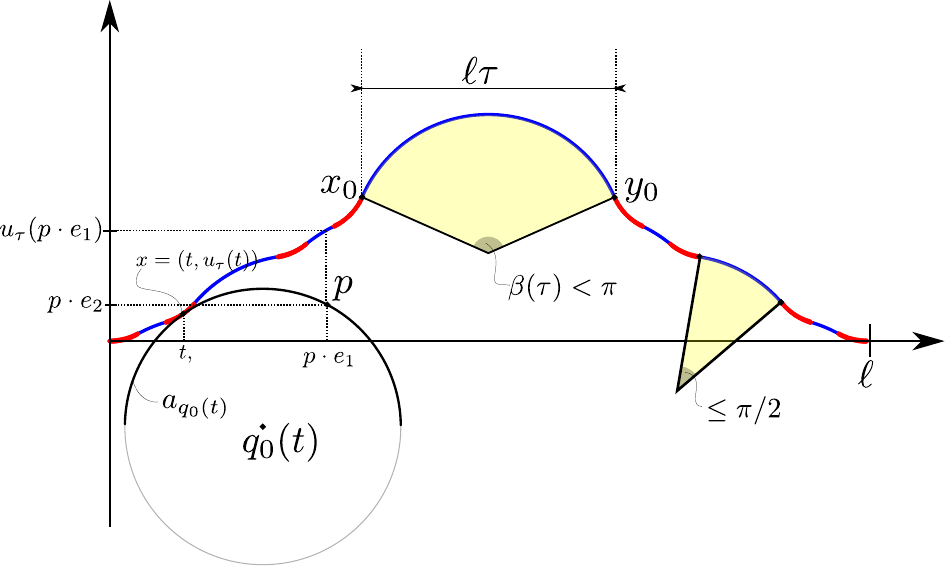}
    \caption{In this picture we represent the analysis of the profile $u_{\tau}$ which has been developed in Proposition \ref{propo:solutionCounter} and Proposition \ref{prop:ontheangle}. We still write $u_{n,\tau}$ in place of $u_{\tau}$ with $n=4$, $\tau=1/3$. The red lines represent the region where $u_{4,1/3}$ fails to solve the constant mean curvature ODE. All the circles depicted have radii $1/H$.}
    \label{fig:PROFILE}
\end{figure}
    \begin{proposition}\label{propo:solutionCounter}
 For any $H\ell<2$, $\tau\in (0,1)$ the function $u_{\tau}(\cdot;H,\ell)$ satisfies the following property: for any $x\in \{(t,u_{\tau}(t;H,\ell))\ | \ t\in (0,\ell)\}$ there exists a unique ball $B_{1/H}$ tangent to $x$ and entirely contained in the epigraph of $u_{\tau}(\cdot;H,\ell)$. 
    \end{proposition}
    \begin{proof}
Once again we omit to write parameters $H,\ell$. Let us pick $t\in (0,\ell)$, let
    \begin{equation}
    q_0(t)=(t,u_{\tau}(t))+\frac{1}{H\sqrt{1+u'_{\tau}(t)^2}}\left( u_{\tau}'(t),-1\right)
    \end{equation}
and consider the ball $B_{1/H}(q_0(t))$. 
Then clearly $(t,u_{\tau}(t))\in \partial B_{1/H}(q_0(t))$. Now we prove that it lies below the graph of $u_{\tau}$. Call 
\[
a_{q_0(t)}:=\{ (r,q_0(t)\cdot e_2+\sqrt{H^{-2}- (r-q_0(t)\cdot e_1)^2}) \ | \ r\in [0,\ell]\cap (q_0(t)\cdot e_1-H^{-1},q_0(t)\cdot e_1+H^{-1}) \}
\]
the upper part of $\partial B_{1/H}(q_0(t))$ which lies in $[0,\ell]\times \mathbb{R}$ (see Figure \ref{fig:PROFILE}). For any $p\in a_{q_0(t)}$ it suffices to prove that
    \[
    p\cdot e_2\leq u_{\tau}(p\cdot e_1)
    \]
which means
    \begin{align*}
        q_0(t)\cdot e_2+\sqrt{H^{-2}- (r-q_0(t)\cdot e_1)^2})\leq u_{\tau}(r)
    \end{align*}
 for all $r\in [0,\ell]\cap (q_0(t)\cdot e_1-H^{-1},q_0(t)\cdot e_1+H^{-1}) $. This becomes
    \begin{align*}
        q_0(t)\cdot e_2+\sqrt{H^{-2}- (r-q_0(t)\cdot e_1)^2})=&u_{\tau}(t)-\frac{1}{H\sqrt{1+(u_{\tau}'(t))^2}}+\sqrt{H^{-2} - \left(r-t-\frac{u_{\tau}'(t)}{H\sqrt{1+(u_{\tau}'(t))^2}}\right)^2}\\
        =&u_{\tau}(t)-\frac{\sqrt{1-(s_{\tau}(t)-Ht)^2}}{H} +\sqrt{H^{-2} - \left(r-t-\frac{s_{\tau}(t)}{H}+t\right)^2}\\
           =&u_{\tau}(t)-\frac{1}{H} \sqrt{1-(s_{\tau}(t)-Ht)^2} +\frac{1}{H}\sqrt{1 - \left(s_{\tau}(t)-Hr \right)^2}.
    \end{align*}
Notice that $(q_0(t)\cdot e_1-H^{-1},q_0(t)\cdot e_1+H^{-1})$ is
    \begin{align*}
        (q_0(t)\cdot e_1-H^{-1},q_0(t)\cdot e_1+H^{-1})&=   \left(\frac{s_{\tau}(t)}{H}-H^{-1},\frac{s(t)}{H}+H^{-1}\right)\\
        &=H^{-1}(s_{\tau}(t)-1,s_{\tau}(t)+1).
    \end{align*}
Thence we need to check that
    \[
    u_{\tau}(t)-\frac{1}{H} \sqrt{1-(s_{\tau}(t)-Ht)^2}  
        +\frac{1}{H}\sqrt{1 - \left(s_{\tau}(t)-Hr \right)^2}\leq u_{\tau}(r).
    \]
for all 
    \[
    r\in (0,\ell)\cap H^{-1}(s_{\tau}(t)-1,s_{\tau}(t)+1).
    \]
    We now invoke Lemma \ref{lem:inequality} and conclude. The uniqueness of the tangent ball comes from the regularity.
    \end{proof}
\begin{proposition}\label{prop:ontheangle} Let $H\ell<2$, $\tau\in (0,1)$, and let $S$ be a connected component of $[0,\ell]\setminus \ell \C(\tau)$ in $[0,\sfrac{(\ell-\ell\tau)}{2}] \cup [\sfrac{(\ell+\ell\tau)}{2},\ell]$. Then the graph of the function $u_{\tau}(\cdot;H,\ell)$ over $S$ is consists of a circular arc of radius $1/H$ spanning an angle smaller than $\pi/2$.\newline \noindent On $(\sfrac{(\ell-\ell\tau)}{2},\sfrac{(\ell+\ell\tau)}{2})$ the graph of the function $u_{\tau}(\cdot;H,\ell)$ is a circular arc of radius $1/H$ and spanning an angle $\beta=\beta(\tau)$ strictly smaller than $\pi$.
\end{proposition}
\begin{proof}
We refer again to Figure \ref{fig:PROFILE} to help the reader in following the proof. It is clear by construction and by means of Proposition \ref{propo:solutionCounter}, that on any connected component of $[0,\ell]\setminus \ell \C(\tau)$ the graph of $u_{\tau}$ is a circular arc of radius $H^{-1}$. The simple fact that it is the graph of a function tells us that the angle spanned by the arc is less than $\pi$. To prove the stronger assertions, first notice that in any region inside $[0,\sfrac{(\ell-\ell\tau)}{2}], [\sfrac{(\ell+\ell\tau)}{2},\ell]$ the angle is easily smaller than $\pi/2$. Indeed let $(a,b)\subset [0,\sfrac{(\ell-\ell\tau)}{2}] $ be a connected component of $[0,\ell]\setminus \ell \C(\tau)$ and notice that if the angle spanned by the circular arc representing $u_{\tau}$ on $(a,b)$ is bigger than $\pi/2$ then we would have $u_{\tau}'(s;H,\ell)=0$ for some $s\in (a,b)$. \newline 
But, from assertion b) of Lemma \ref{lem:propofUn} we have
    \begin{align*}
        u_{\tau}'(t;H,\ell)= \frac{s_{\tau}(t;H,\ell)-Ht}{\sqrt{1-(s_{\tau}(t;H,\ell)-Ht)^2}}>0 \ \ \ \text{on $(0,\sfrac{\ell}{2}) $}.
    \end{align*}
Analogously we argue on connected components of $[0,\ell]\setminus \ell\C(\tau)$ lying in  $[\sfrac{(\ell+\ell\tau)}{2},\ell]$, by exploiting that 
\begin{align*}
        u_{\tau}'(t;H,\ell)= \frac{s_{\tau}(t;H,\ell)-Ht}{\sqrt{1-(s_{\tau}(t;H,\ell)-Ht)^2}}<0 \ \ \ \text{on $(\sfrac{\ell}{2},\ell) $}.
    \end{align*}
Thus we need to check just that the assertion holds for the circular arc lying in $(\sfrac{(\ell-\ell\tau)}{2},\sfrac{(\ell+\ell\tau)}{2})$. To check this we just observe that the chord connecting $x_0=(\sfrac{(\ell-\ell\tau)}{2}, u_{\tau}(\sfrac{(\ell-\ell\tau)}{2};H\ell))$ to $y_0=(\sfrac{(\ell+\ell\tau)}{2}, u_{\tau}(\sfrac{(\ell+\ell\tau)}{2};H\ell)) $ has length  
    \[
    |x_0-y_0|=\frac{2}{H}\sin\left(\frac{\beta}{2}\right)
    \]
being $\beta=\beta(\tau)$ the angle spanned by the arc. But also, since $ u_{\tau}(\sfrac{(\ell-\ell\tau)}{2};H\ell)= u_{\tau}(\sfrac{(\ell+\ell\tau)}{2};H\ell)$ (property c) of Lemma \ref{lem:propofUn}), $|x_0-y_0|=\ell\tau$. In particular, since $H\ell<2$,
    \[
    2\sin\left(\frac{\beta}{2}\right)=H\ell \tau < 2 \tau.
    \]
Thus
\[
\sin\left(\frac{\beta}{2}\right) < \tau<1
\]
and hence $\beta/2<\pi/2$ yielding $\beta<\pi$.
\end{proof}
We collect an easy geometrical fact, that will be useful in the proof of Theorem \ref{thm:sharpness}. 
    \begin{lemma}\label{forsecisifalemma}
    Let $C$ be a circular sector of radius $r$ relative to an arc $a$ spanning an angle $\beta$. Call $x_0$, $y_0$ the left and right extremum of the arc respectively. Let $\vartheta:[0,1]\rightarrow \R^2$ be a curve lying outside of $C$ and such that $\vartheta(0)=x_0$, $\vartheta(1)=y_0$. If $\beta\leq \pi/2$ then for any $x\in C$ we have
        \[
        \dist(x,\vartheta)\leq r.
        \]
    If $\beta \in (\pi/2,\pi)$ then there exists a value $\delta_0=\delta_0(\beta,r)>0$ such that if $\delta<\delta_0$ and $\dist(a,\vartheta)\leq \delta$ then for any $x\in C$
        \[
         \dist(x,\vartheta)\leq r.
        \]
    \end{lemma}
    \begin{proof}
    \begin{figure}
        \centering
        \includegraphics[scale=0.6]{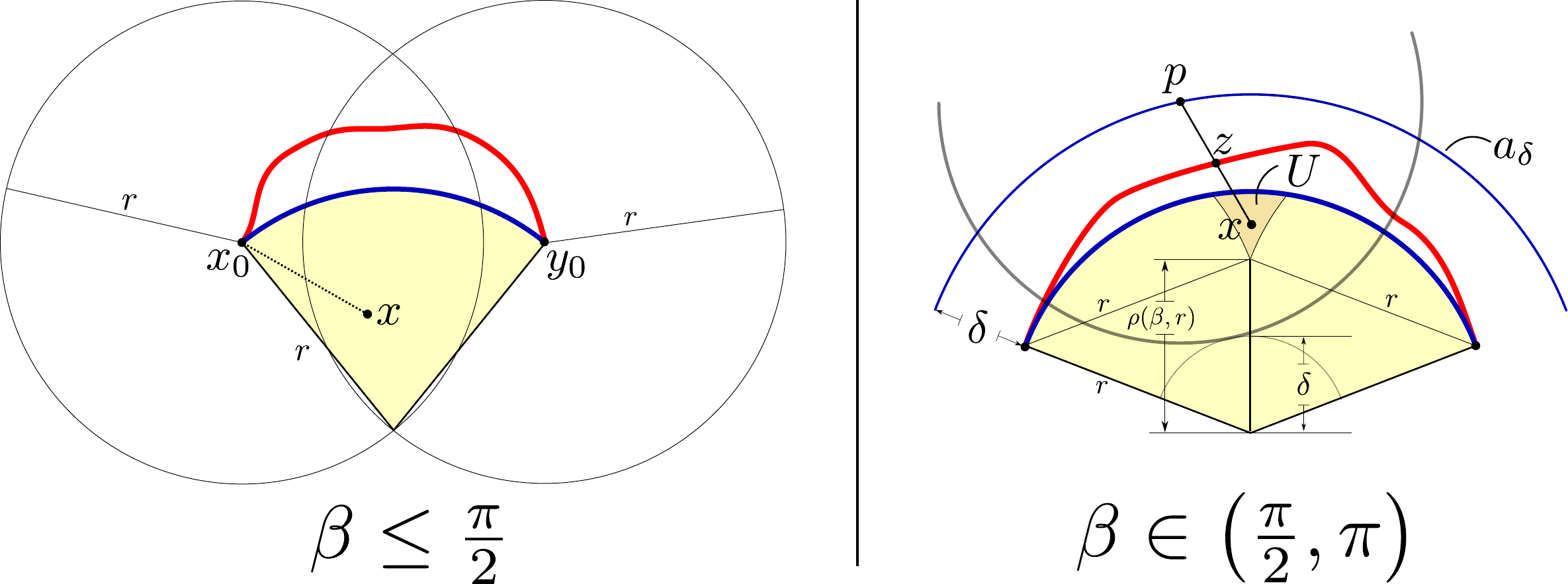}
        \caption{A depiction of Lemma \ref{forsecisifalemma}. The yellow region represents the circular sector $C$, the blue line represents the arc $a$ and the red line the curve $\vartheta$. On the left we show the case in which the circular sector spans an angle smaller than or equal to $\pi/2$, while on the right the case where the circular sector spans an angle greater than $\pi/2$ but strictly less than $\pi$.}
        \label{Figureforsecisifa}
    \end{figure}
We will make use of Figure \ref{Figureforsecisifa} to help the reader in following the proof. We observe that, if $\beta\leq \pi/2$, the union of the two circles centered at $x_0,y_0$ and with radius $r$ covers $C$ (see picture \ref{Figureforsecisifa}). Thus at any $x\in C$ we have
    \[
    \dist(x,\vartheta)\leq \min\{|x-x_0|,|x-y_0|\}\leq r.
    \]
If $\beta\in (\pi/2,\pi)$ call $U:=C\setminus (B_r(x_0)\cup B_r(y_0))\neq \emptyset$. Let $a_{\delta}$ be an arc of radius $r+\delta$ and spanning an angle $\beta$ from  the lines on which $x_0$, $y_0$ lies (see Figure \ref{Figureforsecisifa}). Then $\vartheta$ is forced to lie in between the arc $a$ and the arc $a_{\delta}$. We notice that all balls of radius $r$ centered at a point in $a_{\delta}$ are tangent to a circle of radius $\delta$ centered at the origin. This means that if
    \[
    \delta<\delta_0(r,\beta)< 2r\sin\left(\frac{\pi-\beta}{2}\right)=\varrho(\beta,r),
    \]
then
    \[
    \overline{U}\subset \bigcup_{p\in a_{\delta}} \overline{B_r(p)},
    \]
since $\varrho(\beta,r)$ is the length  of the segment connecting the origin to the intersection point between $\partial B_r(x_0)$ and $\partial B_r(y_0)$ in $\partial U$. Let $x\in \overline{U}$ and let $p\in a_\delta$ be such that $x\in \overline{B_r(p)}$. Then, since $\vartheta$ is connecting $x_0$ to $y_0$, there is a point $z\in \vartheta$ such that the segment connecting $x$ to $p$ intersects $\vartheta$ in $z$. Therefore
    \[
    \dist(x,\vartheta)\leq |x-z|\leq |x-p|\leq r.
    \]
    \end{proof}
\subsection{The construction of the sharp examples}\label{sbsct:Constr}
We are now ready to prove Theorem \ref{thm:sharpness}.
\begin{proof}[Proof of Theorem \ref{thm:sharpness}]
We begin with the Lipschitz case. The $C^{1,\alpha}$ case will follow by a similar argument with the help of Lemma \ref{forsecisifalemma}. In both cases, the main ingredients are Theorems \ref{giogian1}, \ref{giogian2}.\\
\medskip

\textbf{The Lipschitz case.} We will construct a Lipschitz domain such that the Hausdorff dimension of the contact set with its Cheeger set is exaclty a fixed number. Fix this number to be $k\in \N$ and consider the regular $k$-gon $R_k$ of edge $\varrho$. On each edge consider an arc of radius $H^{-1}$ spanning an angle $\beta$ smaller than $\pi$. We refer to Figure \ref{fig:LIP1} for a representation of this situation when $k=6$.
\begin{figure}
    \centering
    \includegraphics{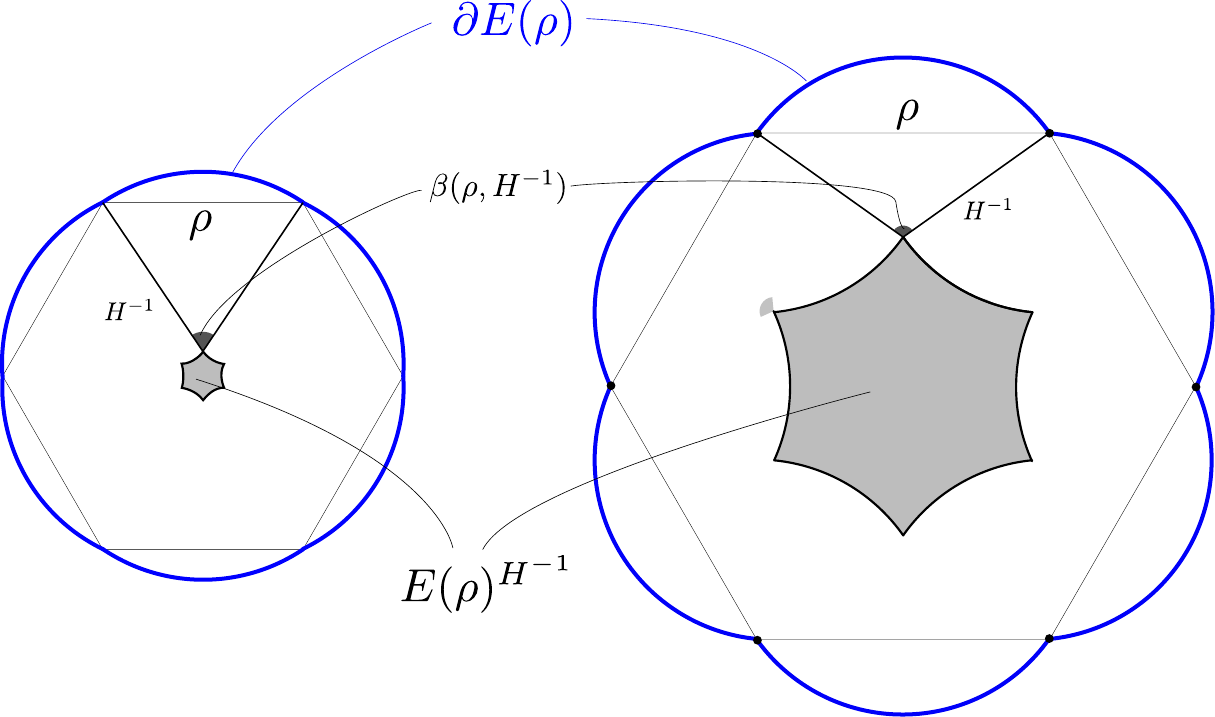}
    \caption{The picture represents the way $\partial E(\varrho)$ is built once $H$ has been fixed. Notice that the smaller is $\varrho$ the smaller is $\L^2(E(\varrho)^{H^{-1}})$.}
    \label{fig:LIP1}
\end{figure}
This procedure is possible with an angle $\beta(\varrho,H)<\pi$, provided $\varrho<\frac{2}{H}$, since clearly
    \begin{equation}\label{pinocchio}
    \varrho=2H^{-1}\sin(\sfrac{\beta}{2}).    
    \end{equation}
Let us set
    \[
   \mathcal{I}_{H}:=\left(0,\frac{2}{H} \right)
    \]
and for $\varrho\in \mathcal{I}_{H}$, let $E(\varrho)$ be the set enclosed by these arcs. Now we show that, for some $\varrho_0\in  \mathcal{I}_{H}$ it holds
    \begin{equation}\label{hh}
    \L^2\left( E(\varrho_0)^{H^{-1}}\right) = \pi H^{-2},
    \end{equation}
being as before
\[
    E(\varrho)^{r}:= \{x\in E(\varrho) \ | \ \dist(x,\partial E(\varrho))>r\}.
    \]
Indeed the map $\varrho\mapsto \L^2\left(E(\varrho)^{H^{-1}}\right)$ is continuous and for small $\varrho$
    \[
    \lim_{\varrho \rightarrow 0 } \L^2 \left(E(\varrho)^{H^{-1}}\right)=0.
    \]
Moreover, as an easy computation shows, we have that
    \[
   \L^2\left(E(\varrho)^{H^{-1}}\right)=\L^2(R_k)-\left(k H^{-1}\frac{\varrho}{2} \cos\left(\frac{\beta}{2}\right)- H^{-2}k\left( \frac{\pi}{2} -\beta\right)\right)+H^{-2}\pi
    \]
and 
    \[
    \L^2(R_k)=\frac{\varrho^2 k}{4\tan(\sfrac{\pi}{k})},
    \]
yielding
    \begin{align*}
     \L^2\left(E(\varrho)^{H^{-1}}\right)&=\pi H^{-2} +\frac{\varrho^2}{4} k\left(\frac{1}{\tan(\sfrac{\pi}{k})}-\frac{1}{\tan(\sfrac{\beta}{2})}+\frac{1}{\sin(\sfrac{\beta}{2})^2}\left(\frac{\pi}{2}-\beta\right)\right)\\
     &=\pi H^{-2} +\frac{\varrho^2}{4} k\left(\frac{1}{\tan(\sfrac{\pi}{k})}-\frac{1}{\tan(\sfrac{\beta}{2})}+\frac{1}{\sin(\sfrac{\beta}{2})^2}\left(\frac{\pi}{2}-\beta\right)\right).
    \end{align*}
For $\varrho \rightarrow 2/H$ \eqref{pinocchio} gives $\beta\rightarrow \pi$ and hence
    \begin{equation}\label{criticalEQ}
    \lim_{\varrho\rightarrow \sfrac{2}{H}} \L^2\left(E(\varrho)^{H^{-1}}\right)=\pi H^{-2} + kH^{-2}\left(\frac{1}{\tan(\sfrac{\pi}{k})}- \frac{\pi}{2}\right).
    \end{equation}
We observe that
    \[
\left(\frac{1}{\tan(\sfrac{\pi}{k})}- \frac{\pi}{2}\right)\geq 0 \ \ \ \Leftrightarrow\ \ \  k\geq \frac{\pi}{\arctan(\sfrac{2}{\pi})}\approx 5.5.
    \]
Hence, for $k\geq 6$ we can achieve also
    \[
    \lim_{\varrho\rightarrow \frac{2}{H}} \L^2\left(E(\varrho)^{H^{-1}}\right)>\pi H^{-2}.
    \]
This means that, for any $k\geq 6$ there exists some $\varrho_0\in (0,\sfrac{2}{H})$ satisfying \eqref{hh}. We also notice that $E(\varrho_0)$ has no neck of radius $H^{-1}$ and that at any $z\in \partial E(\varrho_0)$ there is a ball of radius $H^{-1}$ entirely contained in $E(\varrho_0)$ and tangent at $z$. In particular $E(\varrho_0)= E(\varrho_0)^{H^{-1}} \oplus B_{1/H}$. Therefore we invoke Steiner formulas \eqref{steiner1}, \eqref{steiner2}, that, combined with \eqref{hh}, yield
    \begin{align*}
    \L^2\left( E(\varrho_0) \right)&= \L^2\left( E(\varrho_0)^{H^{-1}}\right)+\mathcal{M}_0\left( E(\varrho_0)^{H^{-1}}\right)H^{-1}+\pi H^{-2}\\
    &=\mathcal{M}_0\left( E(\varrho_0)^{H^{-1}}\right)H^{-1}+2\pi H^{-2}=H^{-1}P(E(\varrho_0)).
    \end{align*}
So, finally $H^{-1}= \frac{\L^2(E(\varrho_0))}{P(E(\varrho_0))}$. Now the very definition of $E(\varrho_0)$, combined with Theorem \ref{giogian1} and the equality $H^{-1}= \frac{\L^2(E(\varrho_0))}{P(E(\varrho_0))}$, implies that $E$ is self-Cheeger and that $h(E(\varrho_0))=H$. 
\begin{figure}
    \centering
    \includegraphics{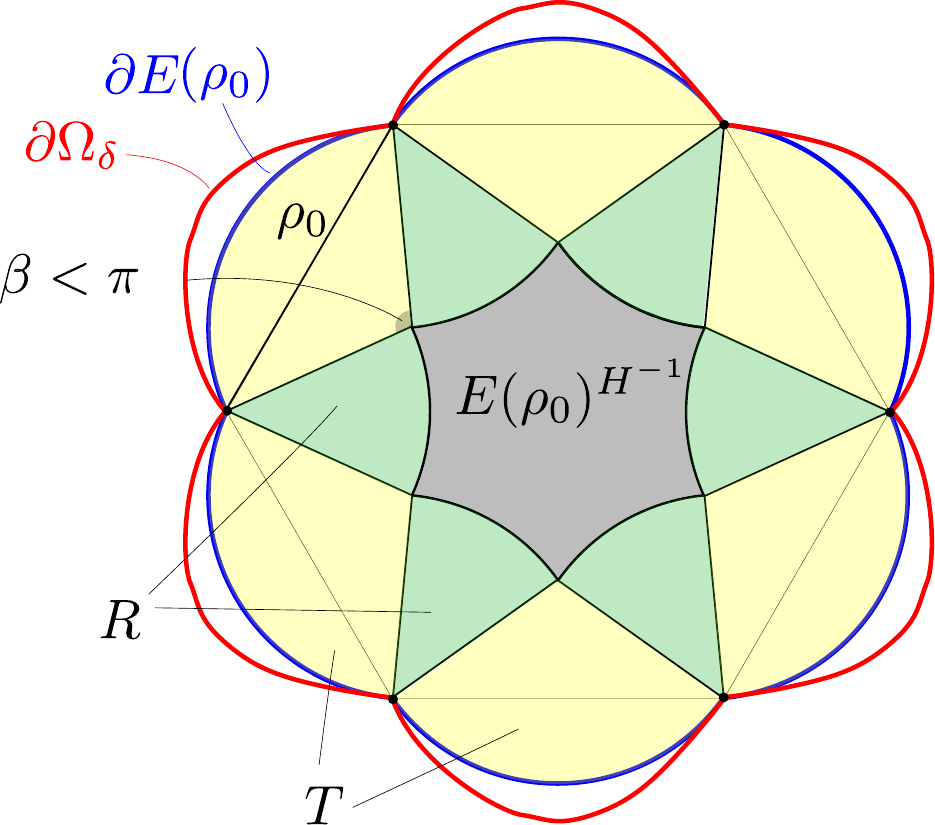}
    \caption{The way $\partial \Omega_{\delta}$ is built once $\varrho_0$ has been calibrated depending on $H$. Notice that all the angles $\beta<\pi$ ensure that we can apply Lemma \ref{forsecisifalemma} and conclude that, for small $\delta$, $E(\varrho_0)$ is a Cheeger set of $\Omega_{\delta}$ as well.}
    \label{fig:LIP2}
\end{figure}
Consider now $\Omega_{\delta}$ to be a small perturbation of $E(\varrho_0)$ in a way that $\partial \Omega_{\delta} \cap \partial E(\varrho_0)=\{\text{the family of vertexes}\}$ and $\dist(\partial \Omega_{\delta},\partial E(\varrho_0))\leq \delta$ (see Figure \ref{fig:LIP2}). We claim that, for some $\delta$ small enough, $\Omega_{\delta}$ has no neck of radius $H^{-1}$ and 
    \begin{equation}\label{e:Lip}
    \Omega_{\delta}^{H^{-1}}=E(\varrho_0)^{H^{-1}}.
    \end{equation}
This will imply, thanks to Theorem \ref{giogian2}, that $E(\varrho_0)$ is a Cheeger set of $\Omega_{\delta}$. Indeed, we know that $h(\Omega_{\delta})^{-1}$ would be the only solution to the equation
    \[
    \L^2(\Omega_{\delta}^r)=\pi r^2
    \]
and \eqref{e:Lip} together with the very definition of $\varrho_0$ would give us that $H^{-1}$ is a solution. Then, by uniqueness we would have $H=h(\Omega_{\delta})=h(E(\varrho_0))$ and $E(\varrho_0)$ Cheeger set of $\Omega_{\delta}$ obtained as $E(\varrho_0)=\Omega_{\delta}^{H^{-1}}\oplus B_{1/H}$.  
The contact set is now given just by the vertexes and thus $\H^{0}(\partial E(\varrho_0)\cap \partial \Omega_{\delta})=k$ with $\Omega_{\delta}$ an open bounded set with Lipschitz boundary.\newline
Since clearly for small values of $\delta$ the set $\Omega_{\delta}$ has no necks of radius $H^{-1}$, we are left to prove \eqref{e:Lip}.\\ \vskip0.1cm \noindent 
Any connected component $S$ of $\partial E(\varrho_0)\setminus \partial \Omega_{\delta}$ is made by an arc $a^S$ spanning an angle strictly smaller than $\pi$. Call $C^S$ the circular sector relative to the arc $a^S$ and split
\[
E(\varrho_0)\setminus E(\varrho_0)^{H^{-1}}=
R\cup T, \ \ \ T=\bigcup_{\substack{\text{$S$ is a connected}\\
    \text{component of}\\
    \text{$\partial  E(\varrho_0)\setminus \partial \Omega_{\delta}$}}} C_S, \ \ \ R= \left( E(\varrho_0)\setminus  E(\varrho_0)^{H^{-1}}\right)\setminus T.
    \]
Clearly, if $z\in \Omega_{\delta}\setminus E(\varrho_0)$, then
    \[
    \dist(z,\partial \Omega_{\delta})<\delta <H^{-1}.
    \]
If instead $z\in E(\varrho_0)\setminus E(\varrho_0)^{H^{-1}}$ then either $z\in T$ or $z\in R$. If $z\in T$ (see again Figure \ref{fig:LIP2}) then it belongs to some $C^S$ relative to an arc $a^S$ spanning an angle less than $\pi.$ In particular by invoking Lemma \ref{forsecisifalemma} we can find a $\delta$ such that for all $z\in C^S$ it holds
    \[
    \dist(z,\partial \Omega_{\delta})\leq H^{-1}.
    \]
Since there are a finite number of arcs we can find a small $\delta$ for which $\dist(z,\partial \Omega_{\delta})\leq H^{-1}$ holds for every $z\in T$. If instead $z\in R$ we simply observe that
    \[
    \dist(z,\partial \Omega_{\delta})=\dist(z,\partial E(\varrho_0)\cap \partial \Omega_{\delta})=\dist(z,\partial E(\varrho_0))\leq H^{-1}.
    \]
Thence \eqref{e:Lip} is in force and the construction of the example is concluded.\\
\medskip

\textbf{The $C^{1,\alpha}$ case.} The logic of the proof is similar to the one in the Lipschitz case. Let us consider the construction of Section \ref{cantor} about the Cantor-type set $\C(\tau)$ and the function $u_{\tau}$. For $\tau=1$ we have $\alpha=\alpha(\tau)=0$ (defined in \eqref{alf}) and for $\tau=0$ we have $\alpha=1$. To produce examples for $\alpha\in (0,1)$ we choose $\tau\in (0,1)$ and fix $H\in \R_+$. Then we consider
    \[
    \mathcal{I}_{H}:=\left(0,\frac{2}{H}\right).
    \]
\begin{figure}
    \centering
    \includegraphics[scale=0.8]{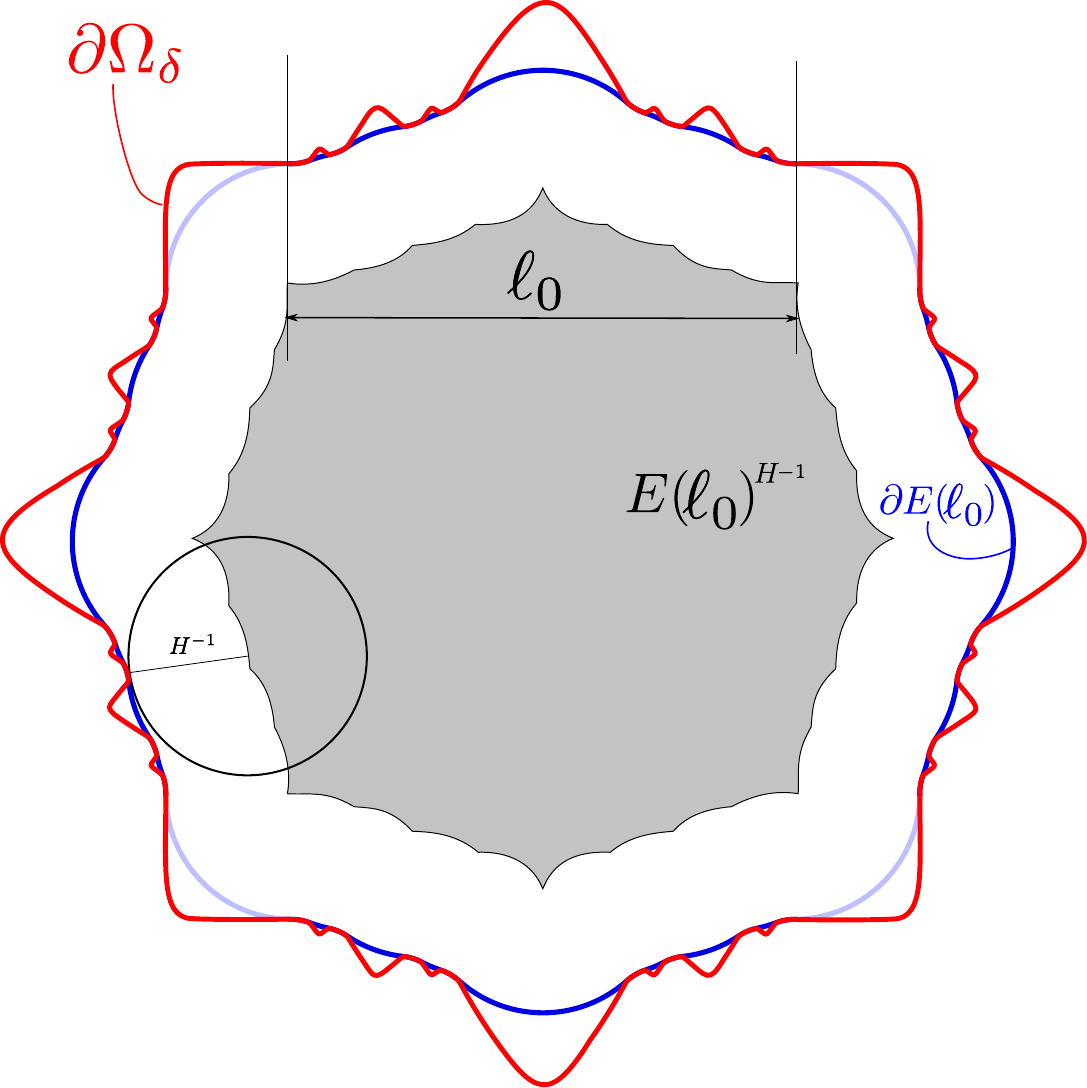}
    \caption{The illustration shows the construction of the sets $\Omega_{\delta} ,E(\ell)$ with the properties inferred by Theorem \ref{thm:sharpness}. To develop the picture we used a profile $u_{n,\tau}$ (for $n=4$, $\tau=1/3$) in place of the profile $u_{\tau}$, as explained in Subsection \ref{sbsct:Constr}. The red line stands for the boundary of the ambient space, while the blue line represents the boundary of $E(\ell_0)$ where $\ell_0$ has been chosen so that \eqref{ggg} is in force. The darker part of $\partial E(\ell_0)$ is the part of the boundary obtained as the graph of $u_{n,\tau}$, while the ligther part are the quarter of circles of radius $1/H$ that has been used to glue together the profiles of $u_{n,\tau}$ in the four directions. The grey set in the middle is the set $E(\ell_0)^{H^{-1}}$ defined by \eqref{e:distset}. By calibrating $\delta$ we can obtain the sought set $\Omega_{\delta}$ with the aid of Theorems \ref{giogian1}, \ref{giogian2}.}
\label{fig:CANTOR1}
\end{figure} \noindent 
For any $\ell\in \mathcal{I}_{H}$ consider a set $E(\ell)$ made by the four copies of $(t,u_{\tau}(t;H,\ell))$, as in Figure \ref{fig:CANTOR1}, joined by four quarter of circle of radius $H^{-1}$. Observe that if we sew the graph of $u_{\tau}$ with a quarter of circle as in Figure \ref{fig:CANTOR1} we preserve the $C^{1,\alpha}$ regularity of the whole profile, due to the fact that $u_{\tau}'(0;H;\ell)=u_{\tau}'(\ell;H,\ell)=0$ and that $u_{\tau}'$ is $C^{\alpha}$ in $(0,\ell)$ as stated in  Lemma \ref{lem:propofUn}. Consider the set
    \begin{equation}\label{e:distset}
    E(\ell)^{H^{-1}}:=\{x\in E(\ell) \ | \ \dist(x,\partial E(\ell))>H^{-1} \}.
 \end{equation}
 We show, as in the Lipschitz case, that there exists an $\ell_0\in \mathcal{I}_{H}$ such that
   \begin{equation}\label{ggg}
    \L^2\left(E(\ell_0)^{H^{-1}}\right)=\pi H^{-2}.
   \end{equation}
In this case the estimates can be easily done by observing that $E(\ell)^{H^{-1}}$ always contains a square of edge length $\ell$. So, for $\ell>\frac{\sqrt{\pi}}{H}$ (which is an admissible value in $\mathcal{I}_{H}$ since $\sqrt{\pi}<2$) we have
    \[
\L^2\left(E(\ell)^{H^{-1}}\right)>\ell^2>\pi  H^{-2}.
    \]
Moreover, any $E(\ell)$ is contained in a ball of radius $\sqrt{2}\ell)$ and thus
    \[
   \L^2\left(E(\ell)^{H^{-1}}\right)\leq 2 \pi \ell^2.  
    \]
In particular for $\ell< \frac{1}{H\sqrt{2}}$ (which is again an admissimble value in $\mathcal{I}_H$) 
we have
        \[
       \L^2\left(E(\ell)^{H^{-1}}\right)< \pi H^{-2}.
        \]
Since the map $\ell \mapsto \L^2\left(E(\ell)^{H^{-1}}\right)$ is continuous the intermediate value theorem tells us that there is an $\ell_0$ such that \eqref{ggg} is satisfied. We now have that $H^{-1}$ satisfies \eqref{ggg}, and $E(\ell_0)$ has no neck of radius $1/H$, since $E(\ell_0)^{H^{-1}}$ is path connected (Remark \ref{senzacollo}). By exploiting Steiner formulas \eqref{steiner1}, \eqref{steiner2} as in the Lipschitz case we can see that $H^{-1}=\frac{\L^d(E(\ell_0))}{P(E(\ell_0))}$. Thus, by applying Theorem \ref{giogian1} and by arguing as in the Lipschitz case (with the aid of Proposition \ref{propo:solutionCounter} in applying Theorem \ref{giogian1}) we can conclude that $E(\ell_0)=E(\ell_0)^{H^{-1}}\oplus B_{1/H}$ is self-Cheeger and that $h(E(\ell_0))=H$.\\ \noindent 
We can modify the ambient space $\Omega_{\delta}$ by gently pushing up $\partial E(\ell_0)$ far away from the Cantor set. Indeed we consider an $\Omega_{\delta}$ represented by the function
    \[
    f_{\delta}(t):=u(t)+\delta g(t) \ \ \ \text{on $[0,\ell]$}
    \]
with $g(t)\in C^{\infty}((0,\ell))$ a non-negative smooth function such that $\ell \C(\tau) = \{g=0\}$,
and a regular small surgery on the part where $E(\ell_0)$ is made by quarter of circles (see again Figure \ref{fig:CANTOR1}). This function has the same regularity of $u$, namely $C^{1,\alpha}$, which yields that $\Omega_{\delta}$ has $C^{1,\alpha}$ boundary regularity. As in the Lipschitz case we just need to prove that, for $\delta$ small enough,
    \begin{equation}\label{finaleq}
   \Omega_{\delta}^{H^{-1}}=E(\ell_0)^{H^{-1}}
    \end{equation}
    and then Theorem \ref{giogian2} will ensure that $E(\ell_0)$ is a Cheeger set of $\Omega_{\delta}$ as well. But now, by construction $\partial E(\ell_0)\cap \partial \Omega_{\delta}$ has the same dimension of the underlying Cantor-type set $\C(\tau)$ used to build the profile $u_{\tau}$ and thus $\dim_{\H}(\partial E(\ell_0)\cap \partial \Omega_{\delta})=\alpha$. This would complete the construction, so we are left with proving \eqref{finaleq}. \vskip0.05cm
    \medskip \noindent
    For $x\in \Omega_{\delta}\setminus E(\ell_0)$ we have, for $\delta$ small enough,
      \[
        \dist(x,\partial \Omega_{\delta})\leq \delta<H^{-1},
        \]
    and thus $E(\ell_0)^{H^{-1}}\subset \Omega_{\delta}^{H {-1}}$. We just need to prove that for $x\in E(\ell_0)\setminus E(\ell_0)^{H^{-1}}$ we have
        \[
        \dist(x,\partial \Omega_{\delta})\leq H^{-1}.
        \]
    \begin{figure}
    \centering
    \includegraphics{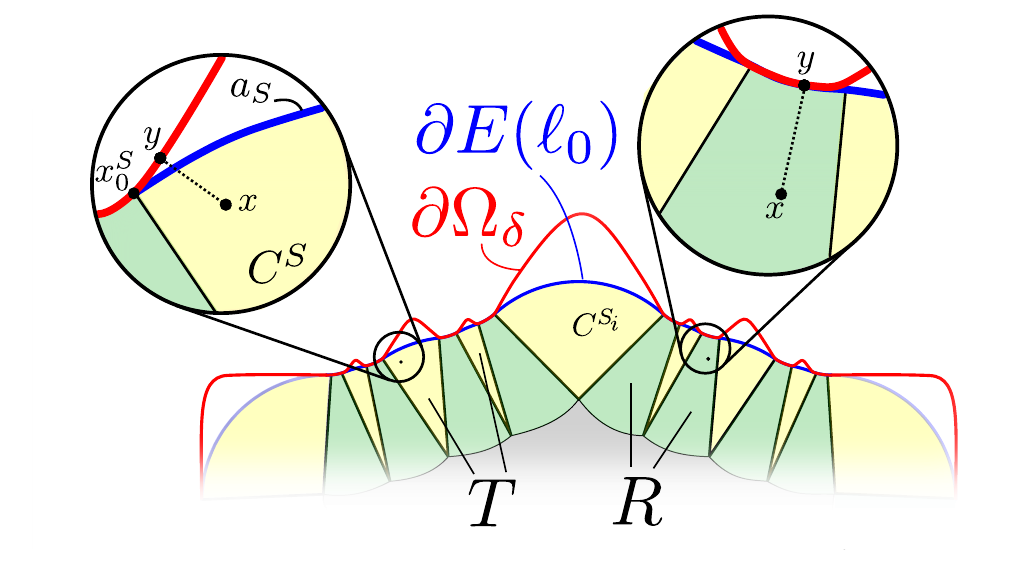}
    \caption{We can split the portion $E(\ell_0)\setminus E(\ell_0)^{H^{-1}}$ into a portion $T$ (in yellow) made by circular sectors $C^S$ of radius $H^{-1}$, corresponding to the connected componend $S$ of $\partial  E(\ell_0)\setminus \partial \Omega_{\delta}$, and a portion $R$ (in green) which is the remaining part. In the case $x\in T$ we are in the situation of Lemma \ref{forsecisifalemma} due to the fact that each $C^S$, but the central ones $C^{S_1},\ldots, C^{S_4}$, do not span an angle bigger than $\pi/2$. The central ones instead span an angle strictly less than $\pi$ and thus for $\delta$ small enough we have $\dist(x,\partial \Omega)\leq H^{-1}$ for $x\in T$. If instead $x\in R$ we simply observe that $|x-y|=\dist(x,\partial\Omega_{\delta})=\dist(x,\partial  E(\ell_0)\cap\partial \Omega_{\delta})=\dist(x,\partial  E(\ell_0))\leq H^{-1}$.}
    \label{figuraSm}
\end{figure}
Propositions \ref{propo:solutionCounter}, \ref{prop:ontheangle} and the construction of $E(\ell_0)$ tell us that on any connected component $S$ of $\partial  E(\ell_0) \setminus \partial \Omega_{\delta}$ is a circular arc $a_S$ from a circle of radius $1/H$ and centered on $\partial E(\ell_0)^{H-1}$ (by definition of $E(\ell_0)^{H^{-1}}$) touching $\partial E(\ell_0)\cap \partial \Omega_{\delta}$ in exactly two points $x_0^S$, $y_0^S$ and entirely lying below $\partial \Omega_{\delta}$ elsewhere. Call $C^S$ the circular sector identified by $a_S$ and notice that we can split $ E(\ell_0) \setminus E(\ell_0)^{H^{-1}}$ as
    \[
     E(\ell_0)\setminus  E(\ell_0)^{H^{-1}}=R\cup T, \ \ \ T=\bigcup_{\substack{\text{$S$ is a connected}\\
    \text{component of}\\
    \text{$\partial  E(\ell_0)\setminus \partial \Omega_{\delta}$}}} C_S, \ \ \ R= \left( E(\ell_0)\setminus  E(\ell_0)^{H^{-1}}\right)\setminus T
    \]
(see Figure \ref{figuraSm}). We consider two cases.\\
\smallskip

\textbf{Case one:} \text{$x\in T$}. In this case we have that $x\in C^S$ for some $S$ connected component of $\partial E(\ell_0) \setminus \partial \Omega_{\delta}$.  Thanks to Proposition \ref{prop:ontheangle} and to the construction of $E(\ell_0)$, that all the circular sectors of $T$, but the four central ones $C^{S_1},\ldots,C^{S_4}$, span an angle smaller than or equal to $\pi/2$. Moreover the remaining sectors $C^{S_1},\ldots,C^{S_4}$ which can span an angle greater than $\pi/2$ spans an angle $\beta_i(\tau)<\pi$ still due to Proposition \ref{prop:ontheangle}. In particular we can invoke Lemma \ref{forsecisifalemma} to conclude that there is a $\delta_0=\delta_0(\beta_1(\tau),\ldots,\beta_4(\tau), H)$ such that if $\delta<\delta_0$
    \[
    \dist(x,\partial \Omega_{\delta})\leq H^{-1}
    \]
and thus $x\notin \Omega_{\delta}^{H^{-1}}$.\\
\smallskip

\textbf{Case two:} \text{$x\in R$}. In this case we immediately have
    \[
    \dist(x,\partial \Omega_{\delta})= \dist(x,\partial E(\ell_0)\cap \partial \Omega_{\delta})=\dist(x,\partial E(\ell_0))\leq H^{-1}
    \]
and then again $x\notin \Omega_{\delta}^{H^{-1}}$.
\end{proof}
\begin{remark}
{\rm 
A careful analysis of the proof of Theorem \ref{thm:sharpness} in the Lipschitz case tells us that we can produce, for any $k\geq 6$, an open bounded set with Lipschitz boundary and having a Cheeger set $E$ such that $\H^0(\partial E\cap \partial \Omega)=k$. Naturally arises the question of whether the value $k=6$ represents some sort of threshold below which we cannot go in performing the construction. In particular, for $k\leq 6$ what fails in the argument is the positivity of the term $\left(\frac{1}{\tan(\sfrac{\pi}{k})}-\frac{\pi}{2}\right)$ in \eqref{criticalEQ}. For $k=3,4,5$ indeed it holds
    \[
    \lim_{\varrho\rightarrow \sfrac{2}{H}}\L^2\left(E(\varrho)^{H^{-1}}\right)<\pi H^{-2}
    \]
and since $\varrho\mapsto\L^2\left(E(\varrho)^{H^{-1}}\right)$ is increasing we would also have 
    \[
    \L^2\left(E(\varrho)^{H^{-1}}\right)< \pi H^{-2}.
    \]
This prevents us from apply Theorems \ref{giogian1}, \ref{giogian2} and we cannot guarantee that $E(\varrho_0)$ will be self-Cheeger for some values of $\varrho_0\in (0,\sfrac{2}{H})$ (actually it will not be self-Cheeger for any value of $\varrho\in (0,2/H)$). It is not the purpose of this analysis to investigate further this question and thus we post-pone the treatment of this topic to future work.
} 
\end{remark}
\section{Appendix}
\addtocontents{toc}{\setcounter{tocdepth}{-10}}
 \subsection*{Proof of Proposition \ref{propo:HausMesChar}}
  \addtocontents{toc}{\setcounter{tocdepth}{1}}

Fix $j \in \N$ and for each $k>j$ let $\mathcal{Q}_k$ be a countable family of cubes of edge $2^{-k}$ yielding the dyadic division of $\R^d$ into a grid. Define for a set $E \subset \R^d$
    \[
    \H_{\star,j}^s(E):=\inf\left\{\left. \sum_{\substack{Q\in \mathcal{Q}_{k}: \\ E\cap Q \neq \emptyset} } 2^{-(k+1)s} \ \right| \ k>j   \right\}
    \]
 and
    \[
   \H^s_{\star}(E):= \lim_{j\rightarrow +\infty} \H_{\star,j}^s(E).
    \]
Notice that the sum is taken over the cubes $Q$ of $\mathcal{Q}_k$ that intersect $E$, i.e. for each choice of $k$ the partition $\mathcal{Q}_k$ changes. It is a well known fact that $\H^s_{\star}$ is a measure on $\R^d$ and is called the \textit{dyadic Hausdorff measure}. Moreover (see for instance \cite{falconer2004fractal} for a full treatment of the whole topic)
    \[
    \H^s(E)\leq \H^s_{\star}(E)\leq C_{d,s} \H^s(E) 
    \]
for a constant $C_{d,s}$ depending on $d,s$ only. Therefore, if $\H^s(N)=0$ then $\H^s_{\star}(N)=0$. In particular, for any $ \varepsilon>0$ there is a $j_{\varepsilon}$ such that
    \[
    \H_{\star,j}^s(N)\leq \varepsilon \ \ \text{for all $j>j_{\varepsilon}$}.
    \]
Thus there exists $k >j_{\varepsilon}$ and a dyadic decomposition $\mathcal{Q}_k$ of $\R^d$ such that
    \[
    \sum_{\substack{Q\in \mathcal{Q}_{k}: \\ N\cap Q \neq \emptyset} } 2^{-(k+1)s}\leq 2\varepsilon.
    \] Relabeling $k$ we conclude the proof.
 \qed

 \subsection{Proof of Lemma \ref{lem:tecnico}}
 We first recall the following result.
 
 \begin{lemma}{\cite[Lemma 17.21]{maggiBOOK}} \label{giulio}
 If $E$ is a set of finite perimeter and $A$ is an open set such that $\H^{d-1}(\partial^*E\cap A )>0$, then there exist $\sigma_0= \sigma_0(E,A)>0$ and $C = C(E, A) <\infty$  such that for every $\sigma \in (-\sigma_0,\sigma_0)$ we can find a set of finite perimeter $F$ with $F\Delta E \subset \subset A$ and
\[
\L^d(F) = \L^d(E) + \sigma,\ \  |P(F;A) -  P(E;A)| \leq C|\sigma|.
\]
 \end{lemma}
\vskip0.2cm \noindent Now we consider $x\in (\partial^* E\cap \Omega)\setminus \Sigma$ such that 
    \begin{itemize}
        \item [a)]$B_{\varrho_0}(x)\subset \subset \Omega\setminus \Sigma$ for some $\varrho_0>0$;
        \item[b)] $\H^{d-1}(\partial^* E \cap B_{\varrho_0}(x))>0 $;
        \item[c)] $\varrho_1=\mathrm{dist}(x,\Sigma)>0$.
    \end{itemize}
Let $\sigma_0=\sigma_0(E,B_{\varrho_0}(x))$, $C=C(E,B_{\varrho_0}(x))$ be the constants given in the Lemma \ref{giulio} above. Let us set $r_0<\min\left\{\left(\frac{\sigma_0 }{\omega_d}\right)^{\sfrac{1}{d}},\varrho_0,\varrho_1\right\}$. Let $ y\in \Sigma $, and let us take $r<r_0$ so that $  \overline{B_r(y)}\cap \overline{B_{\varrho_0}(x)}=\emptyset $. Then we set  
    \[
    \sigma:= \L^d(E)-\L^d(E\setminus B_{r}(y)).
    \]
and notice that 
    \[
   | \sigma|\leq \L^d(B_r(y))\leq \omega_d r^d \leq \omega_d r_0^d\leq\sigma_0. 
     \]
It is possible find a set of finite perimeter $F'$ with $F'\Delta E\subset \subset B_{\varrho_0}(x)$ and such that
    \begin{gather}
\L^d(F')=\L^d(E)+\sigma \label{jack}\\
  |P(F';B_{\varrho_0}(x))-P(E;B_{\varrho_0}(x))|\leq C |\sigma|\label{barbossa}
    \end{gather}
We now define
    \[
    F:= \left((E\setminus B_r(y) ) \cap \overline{B_r(y)} \right) \cup \left(F'\cap \overline{B_r(y)}^c\right)
    \]
and we notice that, since $F'=E$ on $B_r(y)$, by exploiting \eqref{jack} we can achieve
    \begin{align*}
        \L^d(F)
        =\L^d(E).
    \end{align*}
This, combined with the fact that $E$ was a Cheeger set and that by construction $F\subseteq \Omega$, gives
    \[
    \frac{P(E)}{\L^d(E)}\leq \frac{P(F)}{\L^d(F)}= \frac{P(F)}{\L^d(E)}
    \]
which leads us to
    \[
    P(E)\leq P(F).
    \]
We observe that
    \begin{align}
        P(E)&=P\left(E;\overline{B_r(y)}\right)+P\left(E;\overline{B_r(y)}^c\right)\nonumber\\
        &=P\left(E;\overline{B_r(y)}\right)+P\left(E; \left(B_{\varrho_0}(x)\cup \overline{B_r(y)}\right)^c\right)+P(E;B_{\varrho_0}(x))\label{tcha}
    \end{align}
Moreover, since $P\left(F;\overline{B_r(y)}\right)=P\left(E\setminus B_r(y);\overline{B_r(y)}\right)\leq d\omega_d r^{d-1}$ we have the estimate
 \begin{align}
        P(F)&\leq d \omega_d r^{d-1} +P\left(F;\overline{B_r(y)}^c\right)\nonumber\\
        &=d \omega_d r^{d-1}+P\left(F; \left(B_{\varrho_0}(x) \cup \overline{B_r(y)}\right)^c\right)+P(F;B_{\varrho_0}(x))\nonumber\\
        &\leq d \omega_d r^{d-1}+P\left(E; \left(B_{\varrho_0}(x)\cup  \overline{B_r(y)}\right)^c\right)+P(F';B_{\varrho_0}(x))\label{boom}
    \end{align}
since $F=F'=E$ on $\left(B_{\varrho_0}(x)\cup \overline{B_r(y)}\right)^c$ and $F=F'$ on $B_{\varrho_0}(x)$. Then, from \eqref{boom},\eqref{tcha} we have, by exploiting \eqref{barbossa},
    \begin{align}
    P\left(E;\overline{B_r(y)}\right)&\leq  d\omega_d r^{d-1}+P(F';B_{\varrho_0}(x))-P(E;B_{\varrho_0}(x))\nonumber\\
   & \leq  d\omega_d r^{d-1}+C|\sigma|\leq  C_0 r^{d-1}\label{sparrow}
     \end{align}
with $C_0=C_0(x,\varrho_0,\varrho_1,\sigma_0,E)=C_0(\Sigma)$. Since
    \[
      P(E;B_r(y))\leq   P\left(E;\overline{B_r(y)}\right)
      \]
and since \eqref{sparrow} is in force for all $y\in \Sigma$ and for all $r<r_0=r_0(\Sigma)$ we conclude.

\subsection*{Cheeger problem as an obstacle problem}
 Let us briefly treat the obstacle problem of which the graph representation $f_E$ is a solution, for $E$ Cheeger set. Assume that $\partial \Omega \in C^1$ and let $x\in \partial E\cap \partial \Omega$ (assume $x=0$, $\nu_E(x)=\nu_{\Omega}(x)=e_d$). Consider the graph representation of $E$ in $D_r(R)=Q_r\times (-R,R)$, $Q_r\subset \R^{d-1}$, $f_E:Q_r\subset \R^{d-1}\rightarrow (-R,R)$, $r,R>0$, $f_E\in C^1(Q_r)$ (Assertion (III) Theorem \ref{thm:reg}). Then
    \[
    P(E;D_r(R))=\int_{Q_r} \sqrt{1+|\nabla f_E|^2}\d x, \ \ \L^d(E\cap D_r(R))=\int_{Q_r} f_E \, \d x
    \]
Therefore
    \[
    h(\Omega)=\frac{P(E)}{\L^d(E)}=\frac{P(E;D_r(R)^c)+P(E;D_r(R))}{\L^d(E\cap D_r(R)) +\L^d(E\cap D_r(R)^c)}=\frac{P(E;D_r(R)^c)+\int_{Q_r} \sqrt{1+|\nabla f_E|^2}\d x}{\L^d(E\cap D_r(R)^c)+\int_{Q_r}f_E\d x}
    \]
and thus
    \begin{equation}\label{boiadeh}
\L^d(E\cap D_r(R)^c) h(\Omega)=P(E;D_r(R)^c)+\int_{Q_r} (\sqrt{1+|\nabla f_E|^2}-h(\Omega)f_E)\d x.
    \end{equation}
Let $w \in H^1_0(Q_r)$, $w=f_E$ on $\partial Q_r$ and $w\leq f_{\Omega}$ on $Q_r$. We set
    \[
    F_w:=(E\setminus D_r(R)^c) \cup \{(x,s) \ | \ x\in Q_r,  \ s\leq w(x)\}
    \]
Notice that we have still $F_w\subset \Omega$. Moreover $E=F_w$ on $D_r(R)^c$ and
    \[
    P(F_w;D_r(R))= \int_{Q_r} \sqrt{1+|\nabla w|^2}\d x, \ \ \L^d(F_w\cap D_r(R))=\int_{Q_r}w\d x.
    \]
In particular
    \[
   \frac{P(E)}{\L^d(E)}\leq \frac{P(F_w)}{\L^d(F_w)}.
    \]
This yields
    \[
    h(\Omega)\leq \frac{P(E;D_r(R)^c)+\int_{Q_r} \sqrt{1+|\nabla w|^2}\d x}{\L^d(E\cap D_r(R)^c)+\int_{Q_r}w\d x}
    \]
which is, by invoking \eqref{boiadeh},
    \begin{align*}
    \L^d(E\cap D_r(R)^c) h(\Omega) +\int_{Q_r}h(\Omega)w\d x & \leq P(E;D_r(R)^c)+\int_{Q_r} \sqrt{1+|\nabla w|^2}\d x\\
    \int_{Q_r} (\sqrt{1+|\nabla f_E|^2}-h(\Omega)f_E)\d x &\leq \int_{Q_r} (\sqrt{1+|\nabla w|^2}-h(\Omega)w)\d x.
    \end{align*}
This means that $f_E$ solves \eqref{eqn:OBS}. Morever, this minimality property allows us also to conclude that, for all $\varphi\in C^{\infty}_c(Q_r)$, $\varphi\geq 0$, it holds
    \begin{align*}
  \int_{Q_r}  \frac{\nabla u\cdot \nabla \varphi }{\sqrt{1+|\nabla u|^2}} \d x  \leq  
  \int_Q h(\Omega)\varphi \d x.
    \end{align*}
Equivalently, being the above valid for all $\varphi \geq 0$ we have
    \[
    -\dive\left(\frac{\nabla u}{\sqrt{1+|\nabla u|^2}}\right)\leq h(\Omega) \ \ \ \text{weakly on $Q_r$}.
    \]

\bibliography{references}
\bibliographystyle{plain}

\end{document}